\documentclass[reqno,10pt]{amsart}
\usepackage{amsmath}
\usepackage{amsthm}
\usepackage{stmaryrd}
\usepackage{amsfonts}
\usepackage[all]{xy}
\usepackage{mathtools}

\usepackage{verbatim}
\usepackage{amssymb}
\usepackage[english]{babel}
\usepackage{enumerate}
\usepackage{marginnote}
\usepackage{xcolor}

\newcommand{\Val}{\operatorname{Val}}

\providecommand{\hasta}{\longrightarrow}

\newcommand*{\CC}[1]{\C P^{#1}_{\lambda}}
\newcommand*{\RR}[1]{\mathbb{S}^{#1}_{\lambda}}
\newcommand*{\Inv}[2]{\mathcal{I}_{\lambda}^{#1,#2}}
\newcommand*{\derivada}[1]{\partial_{#1}}
\newcommand{\vlr}{\mathcal{V}_{\lambda,\mathbb{R}}^m}
\newcommand{\vlc}{\mathcal{V}_{\lambda,\mathbb{C}}^n}
\usepackage{stmaryrd}
\def\:{\colon}
\def\R{\mathbb{R}}
\def\C{\mathbb{C}}

\def\T{\mathbf{T}}

\def\V{\mathcal{V}}
\providecommand{\nm}[1]{\operatorname{#1}}
\newcommand*{\suce}[1]{\left\lbrace #1 \right\rbrace}
\newcommand*{\restr}[2]{\left.#1\right|_{#2}}
\usepackage{dsfont}
\newtheorem{prop}{Proposition}[section]
\newtheorem{teo}[prop]{Theorem}
\newtheorem*{teosin}{Theorem}
\newtheorem{cor}[prop]{Corollary}
\newtheorem{lema}[prop]{Lemma}
\theoremstyle{definition}
\newtheorem{defi}{Definition}[section]
\theoremstyle{definition}
\theoremstyle{remark}
\newtheorem*{ob}{Remark}

\theoremstyle{definition}
\theoremstyle{definition}
\theoremstyle{remark}
\usepackage{amsmath}
\usepackage{graphicx}

\title{Tube formulas for valuations in complex space forms}

\author{Gil Solanes}
\author{Juan Andrés Trillo}

\email{gil.solanes@uab.cat}
\email{juan.trillo@uab.cat}

\address{Departament de Matem\`atiques, Universitat Aut\`onoma de Barcelona, 08193 Bellaterra, Spain, and Centre de Recerca Matem\`atica, Campus de Bellaterra, 08193 Bellaterra, Spain}
\address{Departament de Matem\`atiques, Universitat Aut\`onoma de Barcelona, 08193 Bellaterra, Spain}

\thanks{Work partially supported by the FEDER/MICINN/AEI grants PGC2018-095998-B-I00, PID2021-125625NB-I00 and the AGAUR grant 2021-SGR-01015.  The first  author is supported by the Serra Hunter Programme and the MICINN/AEI Mar\'ia de Maeztu grant CEX2020-001084-M}

\begin{document}
\begin{abstract}
   Given an isometry invariant valuation on a complex space form we compute its value on the tubes of sufficiently small radii around a set of positive reach. This generalizes classical formulas of Weyl, Gray and others about the volume of tubes. We also develop a general framework on tube formulas for valuations in  riemannian manifolds.
\end{abstract}

\maketitle
\tableofcontents

\section{Introduction}

For a compact convex set $A\subset \R^m$,  the \emph{Steiner formula} computes the volume of the set $A_t$ consisting of points at distance smaller than $t$ from $A$ as follows
\begin{equation}\label{eq:steinerintro}\mathrm{vol}(A_t) = \sum_{i = 0}^m \omega_{m-i}\mu_i(A)t^{m-i}.\end{equation}
Here the functionals $\mu_i$ are the so-called {\em intrinsic volumes}, and the normalizing constant $\omega_k$ is the volume of the $k$-dimensional unit ball. By Hadwiger's characterization theorem, the intrinsic volumes span the space of {\em valuations} (finitely additive functionals on convex bodies) that are  continous and invariant under rigid motions.

The famous {\em tube formula} of H. Weyl is the assertion that \eqref{eq:steinerintro} holds true for $A\subset \R^m$ a smooth compact submanifold and $t\geq0$ small enough, with the additional insight that the coefficients $\mu_i(A)$ depend only on the induced riemannian structure of $A$. Even more generally, Federer extended the validity of \eqref{eq:steinerintro} to the class of compact {\em sets of positive reach}. Later on, the same formula has been proven to hold for bigger classes of sets (see e.g. \cite{FuSubAn, FuPokornyRataj}). As for the coefficients $\mu_i$, the current perspective is to view them as {\em smooth valuations} in the sense of Alesker's theory of valuations on manifolds (see \cite{Aleskersurvey}). 

Already in Weyl's original work, the tube formula was extended to the sphere and to hyperbolic space.  In that case, instead of a polynomial on the radius $t$ one has a polynomial in certain functions $\sin_\lambda(t),\cos_\lambda(t)$ whose definition we recall in \eqref{eq:sin}. Later, Gray and Vanhecke computed the volume of tubes around submanifolds of rank one symmetric spaces (cf. \cite{GrayVanhecke}).

All these classical tube formulas are most naturally expressed in the language of valuations on manifolds. Furthermore, this theory has allowed for the determination of kinematic formulas (a far-reaching generalization of tube formulas) in isotropic spaces. These spaces are riemannian manifolds under the  action of a group of isometries that is transitive on the sphere bundle. For instance, in \cite{BH} and \cite{bfs} the kinematic formulas  of complex complex space forms (i.e. complex euclidean, projective and hyperbolic spaces) were obtained, and Gray's tube formulas on such spaces were recovered.

\medskip Tube formulas, however, exist also for other valuations than the volume, and these do not follow from the kinematic formulas. For instance, differentiating the Steiner formula one easily obtains
\begin{equation}
 \label{eq:steiner_sup}
\mu_k(A_t) = \sum_{j=0}^{k}\binom{m-j}{m-k}\dfrac{\omega_{m-j}}{\omega_{m-k}}\mu_{j}(A) t^{k-j},\qquad A\subset \R^m.
\end{equation}
In real space forms (i.e. the sphere and hyperbolic space), Santaló obtained similar tube formulas for all isometry invariant valuations (see \cite{santalo}). For rank one symmetric spaces, the tube formulas of a certain class of valuations (integrated mean curvatures) were found in \cite{GrayVanhecke}, still with a differential-geometric viewpoint. There are however many invariant valuations on these spaces that were not considered.

\bigskip
In this paper we prove the existence of  tube formulas for any smooth valuation in a riemannian manifold. Then we develop a method to determine these formulas for the invariant valuations of an isotropic space. Using this method we compute all tube formulas explicitly in the case of complex space forms. In fact, our approach also reveals some intersting aspects in the case of real space forms.

Let us briefly describe our results. First, given a riemannian manifold $M$ we construct a family $\mathbf T_t$ of {\em tubular operators} on the space $\mathcal V(M)$ of smooth valuations of $M$ such that for any $\mu\in\mathcal V(M)$ and every compact set of positive reach $A\subset M$ one has
\[
 \mu(A_t)=\mathbf{T}_t\mu(A),
\]
for $t\geq 0$ small enough (see Definition \ref{def:operador_tubular} and Corollary \ref{cor:igualdad_tubos}). Differentiating $\mathbf{T}_t$ at $t=0$ yields an operator $\partial\colon \mathcal V(M)\to\mathcal V(M)$. If $G$ is a group of isometries of $M$ acting transitively on the sphere bundle $SM$, the subspace $\mathcal V(M)^G$ of $G$-invariant valuations is finite dimensional, and the determination of the tube operators $\mathbf{T}_t$ reduces to the computation of the flow generated by $\partial$.

Once this general framework is established we concentrate on the complex space forms $\CC{n}$. For $\lambda=0$ this refers to complex euclidean space $\C^n$ under the group of complex isometries, and for $\lambda\neq 0$ this is the $n$-dimensional complex projective or hyperbolic space of constant holomorphic curvature $4\lambda$, under the full group of isometries $G$. We simply denote $\mathcal V_{\lambda,\C}^n := \mathcal V(\C P^n_\lambda)^G$.

For $\lambda=0$, we will readily obtain the tube formulas $\mathbf{T}_t\mu$ of all translation-invariant and $U(n)$-invariant continuous valuations $\mu$ thanks to the existence of an $\mathfrak{sl}_2$-module structure on the space $\Val^{U(n)}$ of such valuations. This structure,  discovered by Bernig and Fu in \cite{BH}, is induced by two natural operators $\Lambda,L$, the first of which is a normalization of $\partial$.

Remarkably, it turns out that also for $\lambda\neq 0$ the derivation operator $\partial$ is closely related to the operators $\Lambda,L$ of the flat space. Indeed, in Theorem \ref{teo:formulasor} we find an isomorphism $\Phi_\lambda\colon \Val^{U(n)}\to \mathcal V_{\lambda,\C}^n$ such that
\begin{equation}\label{eq:intertwines_intro}
 \left.\partial\right|_{\mathcal V_{\lambda,\C}^n}=\Phi_\lambda\circ (\Lambda-\lambda L)\circ \Phi_\lambda^{-1}.
\end{equation}

Using the decomposition of $\Val^{U(n)}$ into irreducible components, the computation of the tubular operator boils down to the solution of a Cauchy problem in some abstract model spaces, yielding our main result.

\begin{teosin}
 There exists a basis $\{\sigma_{k,r}^\lambda\}$ of the space $\mathcal V_{\lambda,\C}^n$ of invariant valuations of $\CC{n}$ such that
\begin{equation}\label{eq:tube_intro}
\mathbf{T}_t\sigma_{k,r}^{\lambda}=  \sum_{j = 0}^{n}\phi_{2n-4r,k-2r,j}^\lambda(t)\sigma_{j+2r,r}^{\lambda},
\end{equation}
where 
\[
\phi_{m,k,j}^\lambda(t) = \sum_{h \geq 0}(-\lambda)^{j-h}\binom{m-j}{k-h}\binom{j}{h}\sin_\lambda^{k+j-2h}(t)\cos_\lambda^{m-k-j+2h}(t).
\]
\end{teosin}
We describe the basis $\sigma_{k,r}^\lambda$ explicitly in terms of the previously known valuations $\tau_{k,p}^\lambda$ of \cite{bfs}. The tube formulas for the $\tau_{k,p}^\lambda$ can be easily obtained from the previous ones, as we also provide the expression of these valuations in terms of the $\sigma_{k,r}^\lambda$.

\medskip
Curiously, the expressions \eqref{eq:tube_intro} are extremely similar to those obtained by Santaló in the real space form $\RR{m}$ of constant curvature $\lambda$. Indeed, for a certain basis $\{\sigma_i\}_{i=0}^m$ of the space $\mathcal V_{\lambda,\R}^m$ of isometry invariant valuations of $\RR{m}$ one has
\[
\mathbf{T}_t\sigma_i^{\lambda}= \sum_{j = 0}^{m-1}\phi^\lambda_{m-1,i,j}(t)\sigma^\lambda_{j}, \quad 0 \leq i \leq m-1.
\]
The tube formula for $\sigma_m=\nm{vol}$ is however quite different. As an explanation for these similarities, we show in Theorem \ref{teo:formulasor_real} the existence of a phenomenon similar (but not completely analogous) to \eqref{eq:intertwines_intro}.

The paper concludes with a detailed study of the spectrum and the eigenspaces of the derivative operator $\partial$ in $\mathcal V_{\lambda,\C}^n$ and $\mathcal V_{\lambda,\R}^m$. In particular, we compute the kernel of $\partial$ in $\mathcal V_{\lambda,\C}^n$; i.e. we determine the invariant valuations of $\C P_\lambda^n$ for which the tube formulas are constant. We also identify the images $\partial(\mathcal V_{\lambda,\C}^n)$ and $\partial(\mathcal V_{\lambda,\R}^m)$, and we compute the preimage by $\partial$ of any element belonging to these subspaces.

\section{Background}
\subsection{Valuations}
Let $V$ be a finite-dimensional real vector space, and let $\mathcal{K}(V)$ be the space of convex compact subsets of $V$, endowed with the Hausdorff metric. A \emph{valuation} on $V$ is a map $\varphi\colon \mathcal K (V) \to \C$ such that
\[
\varphi(A \cup B)=\varphi(A)+\varphi(B)-\varphi(A \cap B),
\]
for $A,B,A\cup B\in \mathcal{K}(V)$. The space of translation-invariant, continuous valuations on $V$ is denoted by $\Val(V)$.

The notion of valuation was extended to smooth manifolds by Alesker (cf. \cite{AleskerI,AleskerII, AleskerIII,Aleskersurvey}). For simplicity we will focus  on the case of a riemannian manifold $M^n$. 
It is also natural to consider here the class of compact sets of positive reach in $M$, which we denote $\mathcal R(M)$. The definition and some basic properties of such sets are recalled in subsection \ref{subsec:reach}.

Let $SM$ be the sphere bundle of $M$ consisiting of unit tangent vectors, and let $\pi\colon SM\to M$ be the canonical projection. 
\begin{defi}[Smooth valuation]
A \emph{smooth valuation} on $M$ is a functional $\varphi\colon\mathcal{R}(M) \to \C$ of the form
\[
\varphi(A)= \int_{N(A)}\omega + \int_{A}\eta,
\]where $\omega \in \Omega^{n-1}(SM)$ and $\eta \in \Omega^{n}(M)$, are complex-valued differential forms, and $N(A)$ is the normal cycle of $A$ (cf. e.g. \cite{FuSubAn}). We will denote $\varphi=\llbracket \omega,\eta\rrbracket$ in this case. For any subgroup $G \leq \nm{Diff}(M)$, we  will denote by $\mathcal{V}^{G}(M)$ the space of $G$-invariant valuations; i.e. $\mu \in \mathcal{V}(M)$ such that $\mu(gA) = \mu(A)$ for all $A \in \mathcal{R}(M)$ and $g \in G$.
\end{defi}

The kernel of the map $(\omega,\eta)\mapsto \llbracket \omega,\eta\rrbracket$ was determined by Bernig and Br\"ocker in \cite{BernigBroker} as follows. 
Given $\omega \in \Omega^{n-1}(SM)$, there exists $\xi \in \Omega^{n-2}(SM)$ such that
\[
D\omega := d(\omega + \alpha \wedge \xi),
\]{is a multiple of} $\alpha$, the canonical contact form on $SM$.
The unique $n$-form $D\omega$ satisfying this condition is called the \emph{Rumin differential} of $\omega$. Then $\llbracket \omega,\eta \rrbracket = 0$ if and only if
\begin{equation}\label{eq:kernel} 
D\omega + \pi^* \eta = 0, \qquad \mbox{and}\qquad\int_{S_xM}\omega = 0,\quad \forall x \in M. 
\end{equation}

One of the most striking aspects of Alesker's theory of valuations on manifolds is the existence of a natural product on $\mathcal{V}(M)$, which turns this space into an algebra  with $\chi$ as the unit element.  The realization by Fu that this product is closely tied to  kinematic formulas opened the door to the recent development of  integral geometry in several spaces, including the complex space forms \cite{JuditGil,BH,bfs}. \\

Another important algebraic structure is the \emph{convolution of valuations} found by Bernig and Fu in linear spaces (cf. \cite{convolutionconvex}, but also  \cite{convolutionmanifolds}). This is a product on the dense subspace $\Val^\infty(V):=\Val(V)\cap \mathcal V(V)$ characterized as follows. Given $A \in \mathcal{K}(V)$,  with smooth and  positively curved boundary, we have $\mu_A(\cdot) := \nm{vol}(\cdot + A)\in\Val^\infty(V)$. The convolution is determined by  
\begin{equation}\label{eq:convo}
\mu_A * \varphi := \varphi(\cdot + A),\qquad \varphi\in\Val^{\infty}(V),
\end{equation}where $+$ refers to the Minkowski sum. 
In particular, $\nm{vol}$ is the unit element of this operation.

\subsection{Real space forms}

The fundamental examples of  valuations in Euclidean space $\R^m$ are the {\em intrinsic volumes} $\mu_k$. These are implicitly defined by the Steiner formula
\begin{equation}\label{eq:steiner}
\mathrm{vol}_{\R^m}(A + t\mathbb{B}^m) = \sum_{k=0}^m t^{m-k} \omega_{m-k} \mu_k(A), \quad A \in \mathcal{K}(\R^m),
\end{equation}
where $\mathbb{B}^m$ in the unit ball and $\omega_i$ is the volume of the $i$-dimensional unit ball. In particular $\mu_0 = \chi$, $\mu_{m-1} = 2\,\mathrm{perimeter}$, and $\mu_n = \mathrm{vol}_{m}$ are intrinsic volumes.

We will denote by $\RR{m}$ the $m$-dimensional complete and simply connected riemannian manifold of constant curvature $\lambda$. That is, the sphere $S^m(\sqrt{\lambda})$ for $\lambda > 0$, Euclidean space $\mathbb{R}^n$ for $\lambda = 0$, and  hyperbolic space $H^m(\sqrt{-\lambda})$ for $\lambda < 0$. Let $G_{\lambda,\mathbb{R}}$ be the group of {orientation preserving} isometries of $\RR{m}$; i.e. $G_{\lambda,\mathbb{R}}\cong SO(m+1)$ for $\lambda>0$, and $G_{\lambda,\R} \cong SO(m) \rtimes \R^m$ for $\lambda = 0$, while $G_{\lambda,\mathbb{R}}\cong PSO(m,1)$ for $\lambda < 0$. 
We will denote by $\mathcal{V}_{\lambda,\mathbb{R}}^m$ the space of $G_{\lambda,\mathbb{R}}$-invariant valuations of $\mathbb{S}_\lambda^m$.

Let $\kappa_0,\dots,\kappa_{m-1}\in \Omega^{m-1}(S\RR{m})^{{{G}}_{\lambda,\R}}$ be the differential forms defined in \cite[ \S0.4.4]{FuK}.  In the same paper it was shown  that the $\R$-algebra of ${G}_{\lambda,\R}$-invariant differential forms is generated by $\kappa_0,\dots,\kappa_{m-1},\alpha,d\alpha$. It follows by \cite[Prop. 2.6]{bfs} that the following valuations constitute a basis of $\mathcal{V}_{\lambda,\R}^{m}$
\begin{align*}
\sigma_i^{\lambda} &:= \llbracket \kappa_i,0\rrbracket,\qquad 0 \leq i \leq m-1\\
\sigma_m^{\lambda} &:= \nm{vol}_{\RR{m}}.
\end{align*}
In euclidean space $\R^m$ these valuations are proportional to the intrinsic volumes:
\[
\sigma_i^\lambda = {(m-i)\omega_{m-i}}\mu_i,\qquad \lambda=0.
\]
For general $\lambda$, the $\sigma_i^\lambda$ are proportional to the valuations $\tau_i^\lambda$ appearing  in \cite{libroazul,riemcurva}
\begin{align}\label{eq:sigmatau}
\sigma_i^\lambda&=\dfrac{\pi^i(m-i)\omega_{m-i}}{i! \omega_i}\tau_i^{\lambda} , \quad 0 \leq i \leq m-1,\\
\sigma_m^{\lambda}&=\frac{\pi^m}{m!\omega_m}\tau_m^\lambda.\label{eq:sigmatauvol}
\end{align}

As we will see, the normalization taken for the $\sigma_i^\lambda$ makes the tube formulas in $\mathcal V_{\lambda,\R}^m$ specially simple. A stronger reason in favor of this normalization is Theorem \ref{teo:formulasor_real}.

\subsection{Complex space forms}
We denote by $\CC{n}$ the complete, simply connected $n$-dimensional Kähler manifold of constant holomorphic curvature $4\lambda$; i.e. the complex projective space (with the suitably normalized Fubini-Study metric) for $\lambda > 0$, the complex euclidean space $\C^n$ for $\lambda = 0$, and the complex hyperbolic space for $\lambda < 0$. For $\lambda\neq 0$ we let $G_{\lambda,\C}$ be the full isometry group of $\CC{n}$.
For $\lambda=0$ we put $G_{\lambda,\C} = U(n) \rtimes \C^n$. 
We denote by  $\mathcal{V}_{\lambda,\C}^n$ the space of $G_{\lambda,\C}$-invariant valuations on $\CC{n}$. 

Let $\{\beta_{k,q},\gamma_{k,q}\} \subset \Omega^{2n-1}(S\CC{n})^{G_{\lambda,\C}}$ be the differential forms introduced in \cite{BH} for $\lambda=0$, and extended to the curved case $\lambda\neq 0$ in \cite{bfs}. Let also
\begin{align}
&\mu_{k,q}^\lambda := \llbracket \beta_{k,q},0\rrbracket, \quad k>2 q, \label{eq:defmukq} \\
&\mu_{2q,q}^\lambda := \sum_{i = 0}^{n-q-1}\left(\dfrac{\lambda}{\pi}\right)^i \dfrac{(q+i)!}{q!}\llbracket\gamma_{2q+2i,q+i},0\rrbracket + \left(\dfrac{\lambda}{\pi}\right)^{n-q} \dfrac{n!}{q!}\llbracket0,d\mathrm{vol}\rrbracket\label{eq:defmu2qq}
\end{align}where $d\mathrm{vol}$ is the riemannian volume element.
It was shown in \cite{BH,bfs} that these valuations $\mu_{k,q}^\lambda$ with $\max \{0, k-n\} \leq q \leq \frac{k}{2} \leq n$ consitute  a basis of $\mathcal{V}^{n}_{\lambda,\C}$.
It is convenient to emphasize that the $\mu_{k,q}^\lambda$ do \emph{not} coincide with the \emph{hermitian intrinsic volumes}  $\mu_{k,q}^M$ for $M=\CC{n}$ introduced in \cite{Kahlervolume}.

For $\lambda=0$ we simply write $\mu_{k,q}$ instead of $\mu_{k,q}^0$. We will also use the so-called \emph{Tasaki valuations} 
\[
\tau_{k,q}:=\sum_{i=q}^{\lfloor k / 2\rfloor}\binom{i}{q} \mu_{k,i}, \qquad 0, k-n \leq q \leq \frac{k}{2} \leq n.
\]

It will be useful to consider the following linear isomorphisms:
\begin{align*}
\mathcal F_{\lambda,\C}\colon \Val^{U(n)}\hasta \mathcal V_{\lambda,\C}^n,&\quad  \mathcal F_{\lambda,\C}(\mu_{k,q})=\mu_{k,q}^\lambda.
\end{align*}

More generally, whenever we have a valuation $\nu$ in $\Val^{U(n)}$ we will denote $\nu^\lambda := \mathcal{F}_{\lambda,\C}(\nu)$. For instance $\tau_{k,q}^\lambda=\mathcal{F_{\lambda,\C}}(\tau_{k,q})$.

\section{Tube formulas in linear spaces}

Let $V$ be an $m$-dimensional euclidean vector space. Given $t\geq 0$, let $\mathbf{T}_t\colon \Val(V)\to \Val(V)$ be given by 
\begin{equation}\label{eq:def_T_t}
(\mathbf{T}_t\mu)(A)=\mu(A + t\mathbb{B}^m)=(\mu_{t\mathbb{B}^m}*\mu)(A)\qquad A\in\mathcal K(V),
\end{equation}
where $\mathbb{B}^m$ is the unit ball.  We will call $\mathbf{T}_t$ the {\em tubular operator}. Let also $\partial\: \Val(V) \to \Val(V)$ be the operator given by
\begin{equation}
    \partial \mu = \left.\dfrac{d}{dt}\right|_{t = 0}\mathbf{T}_t\mu.
\end{equation}
This operator has sometimes been denoted by $\Lambda$ in the literature, but following \cite{BH} we reserve the symbol $\Lambda$ for a certain normalization of $\partial$ (see \eqref{eq:def_Lambda_L_H}).

The properties of the Minkowski sum ensure that $\mathbf{T}_{t+s} = \mathbf{T}_t \circ \mathbf{T}_s=\mathbf{T}_s \circ \mathbf{T}_t$. Differentating with respect to $s$ at zero yields
\begin{equation}\label{eq:problemacauchyplano}
    \dfrac{d}{dt}\mathbf{T}_t\mu= \mathbf{T}_t\partial\mu = \partial \mathbf{T}_t\mu.
\end{equation}
It follows that
\begin{align}\label{eq:partialiderivada}
    \partial^i\mu = \left.\dfrac{d^i}{dt^i}\right|_{t = 0}\mathbf{T}_t\mu.
\end{align}
For each $\mu \in \Val(V)$, the map $t \mapsto \mathbf{T}_t\mu$ is a polynomial in $t$ of degree $m$ by \eqref{eq:def_T_t} and the Steiner formula \eqref{eq:steiner} (or by \cite{mcmullen}). Hence 
\begin{align}\label{eq:tubularplanotaylor2}
\mathbf{T}_t\mu &= \sum_{i = 0}^m \dfrac{t^i}{i!}\left.\dfrac{d^i}{dt^i}\right|_{t = 0}\mathbf{T}_t\mu\\
\label{eq:tubularplanotaylor}
&= \sum_{i = 0}^m \dfrac{t^i}{i!}\partial^i \mu.
\end{align}
Note also that, by \eqref{eq:partialiderivada} and \eqref{eq:tubularplanotaylor2}, the derivative operator $\partial$  is $(m+1)$-nilpotent; i.e. $\partial^{m+1} = 0$.

Let us compute the tube formula for the intrinsic volume $\mu_i$ for each $0 \leq i \leq m$ using \eqref{eq:tubularplanotaylor}. For that purpose we first compute $\partial$. Since $\mathbf{T}_{t+s} = \mathbf{T}_s\circ \mathbf{T}_t$ we have\[
\mathbf T_{t+s}\nm{vol} = \sum_{j = 0}^m \omega_{m-j}t^{m-j}\mathbf T_s\mu_j.
\]
On the other hand
\[
\mathbf T_{t+s}\nm{vol} = \sum_{j = 0}^m \omega_{m-j}(t+s)^{m-j}\mu_j,
\]
Differentiating at $s=0$ and comparing coefficients yields
\[
\partial\mu_j = \dfrac{\omega_{m-j+1}}{\omega_{m-j}}(m-j+1)\mu_{j-1}.
\]
Finally, using \eqref{eq:tubularplanotaylor}, we get
\begin{align*}
\mathbf{T}_t\mu_k & = \sum_{i = 0}^m \frac{t^i}{i!}\partial^i\mu_k =\sum_{i = 0}^k \frac{t^i}{i!}\frac{\omega_{m-k+i}}{\omega_{m-k}}\frac{(m-k+i)!}{(m-k)!}\mu_{k-i}\\
&= \sum_{j=0}^k \binom{m-j}{k-j}\frac{\omega_{m-j}}{\omega_{m-k}} t^{k-j}\mu_j,
\end{align*}
which is \eqref{eq:steiner_sup}.

\bigskip
In order to compute the tube formulas for invariant valuations in $\C^n$ (i.e. to determine $\mathbf T_t$ on $\Val^{U(n)}$), it will be useful to recall the $\mathfrak{sl}_2$-module structure of $\Val^{U(n)}$ found in \cite{BH}.
Consider the linear maps $\Lambda, L, H\colon\Val(V)\to\Val(V)$, defined as follows
\begin{equation}\label{eq:def_Lambda_L_H}
\Lambda \nu := \dfrac{\omega_{m-k}}{\omega_{m-k+1}}\derivada{}\nu,\qquad L \nu := \dfrac{2\omega_k}{\omega_{k+1}} \mu_1 \cdot \nu,\qquad H\nu=(2k-m)\nu,\qquad  \nu\in\Val_k(V),
\end{equation}where $\cdot$ refers to the Alesker product.

\begin{prop}\label{propformulas}
On $\Val^{O(m)}$ the operators $\Lambda,L$ are given by
\begin{align}
&L \mu_k = (k+1)\mu_{k+1},\label{eq:Lmu} \\
&\Lambda \mu_k = (m-k+1)\mu_{k-1}\label{eq:Lambdamu},
\end{align}
while on $\Val^{U(n)}$ one has
\begin{align}
 &L \mu_{k, p} =(k-2 q+1) \mu_{k+1, q}+2(q+1) \mu_{k+1, q+1}\label{eq:L_mu} \\
&\Lambda_{} \mu_{k, p}= (k-2 q+1) \mu_{k-1, q-1}+2(n-k+q+1) \mu_{k-1, q},\label{eq:Lambda_mu}
\end{align}
which implies
\begin{align}
&L\tau_{k,q} = (k-2q+1)\tau_{k+1,q}\label{eq:Ltau}\\
&\Lambda\tau_{k,q} = (k-2 q+1) \tau_{k-1, q-1}+(2 n-2 q-k+1) \tau_{k-1, q}
\end{align}
\end{prop}

\begin{proof}
    The first two equalities are \cite[eqs. (2.3.12) and (2.3.13)]{libroazul}. The rest is  \cite[Lemma 5.2]{BH}.
\end{proof}

\begin{prop}[{{\cite[Prop. 2.3.10 (3)]{libroazul}}}]
The operators $\Lambda, L,H$ define an $\mathfrak{sl}_2$-module structure on both $\Val^{O(m)}$ and $\Val^{U(n)}$; i.e. $[L,\Lambda]=H$, $[H,L]=2L, [H,\Lambda]=-2\Lambda$.
\end{prop}

The decomposition into irreducible components is as follows
\begin{equation}\label{eq:sl2_decomposition}
{\Val^{O(m)}\cong V^{(m)}},\qquad \Val^{U(n)}\cong \bigoplus_{0 \leq 2r\leq n}V^{(2n-4r)}
\end{equation}where $V^{(m)}$ is the $(m+1)-$dimensional irreducible $\mathfrak{sl}_2$-representation.
In particular, for $0 \leq 2r \leq n$, there exists a unique, up to a multiplicative constant, primitive element  (i.e. anihilated by $\Lambda$) in each irreducible component of $\Val^{U(n)}$. By the so-called Lefschetz decomposition, the $L$-orbits of these primitive elements consitute a basis of $\Val^{U(n)}$. This basis was explictly computed in \cite{BH} as follows.

\begin{prop}[{{\cite[eq.\,(76)]{BH}}}]\label{prop:primitives}
The following valuations 
\begin{align}\label{eq:pi2rr}
\pi_{2r,r} &:= (-1)^{r}(2 n-4 r+1) ! ! \sum_{i=0}^{r}(-1)^{i} \frac{(2 r-2 i-1) ! !}{(2 n-2 r-2 i+1) ! !} \tau_{2 r, i}, & 0 \leq 2r \leq n,
\end{align}are $\Lambda$-primitive; i.e. $\Lambda\pi_{2r,r}=0$. The family
\begin{align}
\pi_{k, r} &:=L^{k-2 r} \pi_{2 r, r} \\
&=(-1)^{r}(2 n-4 r+1) ! ! \sum_{i=0}^{r}(-1)^{i} \frac{(k-2 i) !}{(2 r-2 i) !} \frac{(2 r-2 i-1) ! !}{(2 n-2 r-2 i+1) ! !} \tau_{k, i}, & 2r \leq k \leq 2n-2r \label{eq:pikrtasakis}
\end{align}
forms a basis of $\nm{Val}^{U(n)}$.
\end{prop}

In particular the irreducible components of $\Val^{U(n)}$ are the following subspaces
\begin{equation}\label{eq:subespaciosinvariantesplanos}
\mathcal{I}_{0}^{n,r} := \suce{\pi_{k,r} : 2r \leq k \leq 2n-2r}, \quad 0 \leq 2r \leq n.
\end{equation}

We are now able to compute the tube formulas in the complex case using \eqref{eq:tubularplanotaylor}.

\begin{teo}\label{teo:formulasplanas}
\begin{align}\label{eq:tubo_pikr}
\mathbf{T}_t\pi_{k,r} &= \dfrac{(k-2r)!}{\omega_{2n-k}}\sum_{j = 0}^{k-2r} \binom{2n-4r-j}{k-2r-j}t^{k-2r-j}\dfrac{\omega_{2n-2r-j}}{j!}\pi_{j+2r,r}.
\end{align}
\end{teo}

\begin{proof}
By \cite[Lemma 5.6]{BH}, 
\begin{equation}\label{eq:Lambdapikrplana}
\Lambda \pi_{k,r} = (k-2r)(2n-k-2r+1)\pi_{k-1,r},
\end{equation}
and then
\begin{align}\label{eq:Lambdapikrplanai}
\Lambda^{i} \pi_{k,r} &= \dfrac{(k-2r)!(2n-k-2r+i)!}{(k-2r-i)!(2n-k-2r)!}\pi_{k-i,r}.
\end{align}
Using \eqref{eq:tubularplanotaylor}, we obtain the tube formula
\begin{align*}
    \mathbf{T}_t\pi_{k,r} &= \sum_{i = 0}^{2n} \dfrac{t^i}{i!} \frac{\omega_{2n-k+i}}{\omega_{2n-k}}\Lambda^i\pi_{k,r}\\& =\dfrac{(k-2r)!}{\omega_{2n-k}}\sum_{i = 0}^{k-2r} \dfrac{t^i}{i!}\omega_{2n-k+i}\dfrac{(2n-k-2r+i)!}{(k-2r-i)!(2n-k-2r)!}\pi_{k-i,r} \\
    &=\dfrac{(k-2r)!}{\omega_{2n-k}}\sum_{j = 0}^{k-2r} \binom{2n-4r-j}{k-2r-j}t^{k-2r-j}\dfrac{\omega_{2n-2r-j}}{j!}\pi_{j+2r,r}.\qedhere
\end{align*}
\end{proof}

These tube formulas can also be given in terms of the valuations $\tau_{k,q}$. To this end, we next compute their \emph{Lefschetz decomposition}. 

\begin{prop}
The Lefschetz decomposition of $\tau_{k,r}$ is given by
\begin{align}\label{eq:descLef}
\tau_{k,r} &= \frac{1}{(k-2r)!} \sum_{i = 0}^r \binom{n-2i}{r-i}\frac{(2n-2i-2r) !}{(2 n-4 i) !} \pi_{k,i}.
\end{align}
\end{prop}
\begin{proof}

Consider the linear map $\psi : \Val^{U(n)} \to \Val^{U(n)}$ mapping $\tau_{k,r}$ to the left hand side of \eqref{eq:descLef}.
We need to show that $\psi=\nm{id}$. Let us check that this endomorphism commutes with both  $\Lambda$ and $L$.
To check commutation with $\Lambda$, we only need to verify the following 
    \begin{align*}
    (k-2r)!\psi(\Lambda(\tau_{k,r}))  =&\sum_{i=0}^{r-1} \frac{(n-2i) !(2n-2i-2r+2)!}{(r-i-1) !(n-r-i+1)!(2n-4i) !}\pi_{k-1, i}\\
    &+(k-2r)(2n-k-2r+1)\sum_{i=0}^r \frac{(n-2 i) !(2n-2i-2r)!}{(r-i)!(n-i-r)!(2n-4i)!}\pi_{k-1, i}\\
    =&\sum_{i = 0}^r \dfrac{(n-2i)!(2n-2i-2r)!}{(r-i)!(n-i-r)!(2n-4i)!}(k-2i)(2n-k-2i+1)\pi_{k-1,i} \\
    =&(k-2r)!\Lambda\psi(\tau_{k,r}).
    \end{align*}
Comparing term by term, the previous identities boil down to    
\[
2 (r - i) (2 n - 2 i - 2 r + 1) + (k - 2 r) (2 n - k - 2 r + 
    1) = (k - 2 i) (2 n - k - 2 i + 1) 
\] 
which is trivial.

 Commutation with $L$ is straightforward using $L\pi_{k,i} = \pi_{k+1,i}$.

Given that $\psi$ commutes with the operators $\Lambda$ and $L$ and  $\Val^{U(n)}$ is multiplicity-free, Schur's lemma implies that for each $0 \leq 2r \leq n$, there exists a constant $c_r$ such that $\restr{\psi}{\mathcal{I}_{0}^{n,r}} = c_r \nm{id}$. 

    Let $a_{2r,j}$ and $b_{2r,i}$ be the coeficients of $\pi_{2r,j}$ and $\tau_{2r,i}$ in \eqref{eq:descLef} and \eqref{eq:pi2rr} respectively, so that $\psi(\tau_{2r,i}) = \sum_{j= 0}^i a_{2r,j} \pi_{2r,j}$ and $\pi_{2r,r} = \sum_{i = 0}^r b_{2r,i} \tau_{2r,i}$. Then
    \[
    c_r \pi_{2r,r} = \psi(\pi_{2r,r})=\sum_{i = 0}^rb_{2r,i}\left(\sum_{j = 0}^{i} a_{2r,j} \pi_{2r,j}\right)=\sum_{j = 0}^{r}\sum_{i = j}^r b_{2r,i}a_{2r,j} \pi_{2r,j}.
    \]
    Comparing the coefficient of $\pi_{2r,r}$ on both sides we get $c_r=b_{2r,r}a_{2r,r}=1$ for each $0 \leq 2r \leq n$. Hence $\psi = \nm{id}$, which proves \eqref{eq:descLef}.\qedhere
\end{proof}

By plugging  \eqref{eq:pikrtasakis} and \eqref{eq:descLef} in \eqref{eq:tubo_pikr} one gets the tube formulas $\mathbf{T}_t\tau_{k,p}$ in terms of the $\tau_{i,j}$.

\section{Tube formulas in riemannian manifolds}
\subsection{Tubular and derivative operators}

Next we extend to any complete riemannian manifold $M$ the tubular operator $\mathbf T_t$ introduced in the previous section on linear spaces.
Let $T$ be the \textit{Reeb vector field} on $SM$, which is characterized by $i_T \alpha = 1$ and $\mathcal{L}_T\alpha = 0$, where $\mathcal{L}$ is the Lie derivative. The Reeb flow $\phi: SM \times \R \to SM$, defined as the flow of $T$, is a family of contactmorphisms and coincides with the geodesic flow on $SM$ (see e.g. \cite[Theorem 1.5.2]{contacttopology}).

\begin{defi}[Tubular and derivative operators]\label{def:operador_tubular}
Given $t\geq0$, we define the \emph{tubular operator} $\mathbf{T}_t$ by
\[
\mathbf{T}_t \: \mathcal{V}(M) \hasta \mathcal{V}(M), \quad \llbracket \omega,\eta\rrbracket \longmapsto \llbracket \phi_t^{*}\omega + (p_t)_*(\pi \circ \phi)^*\eta , \eta \rrbracket,
\]where $p_t \: SM \times [0,t] \to SM$ is the projection on the second factor, and $\phi_t=\phi(\cdot,t)$.
We define the \emph{derivative operator} $\partial=\partial_M$ by
\[
\derivada{M} \: \mathcal{V}(M) \hasta \mathcal{V}(M), \quad \mu \longmapsto \left. \dfrac{d}{dt}\right|_{t = 0} \mathbf{T}_t\mu.
\]
\end{defi}

To show that these definitions are consistent, suppose $\mu = \llbracket \omega,\eta\rrbracket = 0$, and let us check that $\mathbf{T}_t\mu = 0$ for all $t \geq 0$, i.e.
\[
\int_{N(A)}\phi_t^{*}\omega + \int_{N(A)} (p_t)_*(\pi \circ \phi)^*\eta + \int_A \eta = 0, \quad \forall A \in \mathcal{R}(M).
\]
By \eqref{eq:kernel} we have $\pi^*\eta = -D\omega=-d(\omega+\xi\wedge\alpha)$. Hence
\[
\begin{aligned}
\int_{N(A)} (p_t)_*(\pi \circ \phi)^*\eta &= - \int_{N(A)}(p_t)_*\circ \phi^*D\omega = - \int_{N(A)\times [0,t]}\phi^* d(\omega+
\xi \wedge \alpha) \\
&=- \int_{N(A) \times [0,t]}d\phi^{*}(\omega+\xi\wedge\alpha) = -\int_{N(A)\times \{0,t\}}\phi^* \omega+\phi^*\xi\wedge\alpha \\
&= \int_{N(A)}\phi_{0}^*\omega  - \int_{N(A)}\phi_t^{*}\omega = \int_{N(A)}\omega - \int_{N(A)}\phi_t^{*}\omega,
 \end{aligned}
\]
as $\alpha$ vanishes on $N(A)$.
Since $\llbracket\omega,\eta\rrbracket = 0$, we have $\int_{N(A)}\omega = -\int_A\eta$. Therefore $\mathbf{T}_t \mu = 0$. 

Let us next establish some basic properties of these operators.

\begin{lema}\label{lema:derivada_dificil}
    \[
    \frac{d}{dt}(p_t)_*\phi^*\rho =i_T \phi_t^*\rho,\qquad \rho\in \Omega^*(SM)
    \]
\end{lema}

\begin{proof}
    Given a compact smooth submanifold $N\subset SM$, 
\begin{align*}
   \int_N (p_t)_*\phi^*\rho&=\int_{N\times [0,t]} \phi^{*}\rho\\
   &=\int_{N\times [0,t]} i_{\frac\partial{\partial t}}\phi^{*}\rho\wedge dt\\
   &= \int_0^t\left(\int_{N}  \phi_t^{*}i_{\frac{\partial\phi}{\partial t}}\rho\right) dt, \\
\end{align*}
Since $i_T$ and $\phi_t^*$ commute, the result follows. 
\end{proof}

\begin{prop}\label{prop:derivada}
For $\mu = \llbracket \omega,\eta\rrbracket$,
\begin{enumerate}[i)]
    \item $\partial \mu = \llbracket i_T \left(d\omega + \pi^*\eta\right),0\rrbracket$
    \item $\mathbf{T}_{t+s}\mu=(\mathbf T_{t}\circ \mathbf T_{s})\mu$
\end{enumerate}
\end{prop}
\begin{proof}
Modulo exact forms we have
\begin{equation}\label{eq:lie_der}
\frac{d}{dt} \phi_{t}^*\omega=\left.\frac{d}{ds}\right|_{s=0} \phi_{t+s}^*\omega=\mathcal L_T \phi_t^*\omega\equiv i_T \phi_t^*d\omega.
\end{equation}
Together with Lemma \ref{lema:derivada_dificil}, this yields
\begin{equation}\label{eq:derivada_T}
    \frac{d}{dt}\mathbf T_t \mu=\llbracket i_T(d\phi_t^*\omega+\phi_t^*\pi^*\eta),0\rrbracket
\end{equation}
Evaluating at $t=0$, this gives $i)$.

In order to prove $ii)$, it is enough to check that both sides have the same derivative with respect to $s$, as they clearly agree for $s=0$.
By \eqref{eq:derivada_T}, we have
\begin{align*}
    \frac{d}{ds}\mathbf T_t\circ \T_s (\mu)&= \mathbf T_t\circ  \frac{d}{ds} \T_s(\mu) \\
    &=\mathbf T_t \llbracket i_T(d\phi_s^*\omega+\phi_s^*\pi^*\eta),0\rrbracket\\
    &=  \llbracket \phi_t^* i_T(d\phi_s^*\omega+\phi_s^*\pi^*\eta),0\rrbracket.
\end{align*}
Since $\phi_t^*$ and $i_T$ commute, it follows from \eqref{eq:derivada_T} that $\frac{d}{ds}\mathbf{T}_{t+s}=\frac{d}{ds}\mathbf T_{t}\circ \mathbf T_{s}$.
\end{proof}

Fix $\mu \in \mathcal{V}(M)$. It follows from Proposition \ref{prop:derivada} $ii)$ that
\begin{equation}\label{eq:Cauchy_Tmu}
\frac{d}{dt} \mathbf{T}_t\mu=\derivada{} \mathbf{T}_t\mu.
\end{equation}
If $\mu \in \mathcal V(M)^G$ for a group $G$ acting on $M$ by isometries, then also $\mathbf{T}_t\mu\in\mathcal V(M)^G$. Hence, in case $\mathcal V^G(M)$ is finite-dimensional,  computing $\mathbf{T}_t\mu$ boilws down to solving the Cauchy problem \eqref{eq:Cauchy_Tmu} with initial condition $\mathbf{T}_0\mu=\mu$; i.e. 
\begin{equation}\label{eq:exponential_T}
\mathbf{T}_t\mu = \nm{exp}(t\partial)\mu=  \sum_{i \geq 0}\dfrac{t^i}{i!}\partial^i\mu.
\end{equation}
This is the approach we will follow to obtain the tube formulas for invariant valuations in complex space forms.
Note that \eqref{eq:exponential_T} coincides with \eqref{eq:tubularplanotaylor2} except that $\partial$ does not need to be nilpotent for general $M$.
\subsection{Tubes in riemannnian manifolds}\label{subsec:reach}

Let $M$ be a complete riemannian manifold and let $d\: M \times M \to [0,\infty)$ be the riemannian distance on $M$. For $t\geq 0$, the \emph{tube} of radius $t$ around a subset  $A\subset M$ is defined as
\[
A_t:= \suce{p \in M : d_A(p) \leq t},
\]
where
\[
d_A(p) := \inf\suce{d(p,q) : q \in A}.
\]
Next we review some basic facts about tubes around sets of positive reach (introduced by Federer in euclidean spaces and by Kleinjohann in riemannian manifolds). For such sets $A$ we will prove that  $\mathbf{T}_t \mu(A)=\mu(A_t)$ for any $\mu\in\mathcal V(M)$ and sufficiently small $t$.

\begin{defi}[Sets of positive reach]
    A {\em set of positive reach} in $M$ is a closed  subset $A\subset M$ for which there exists an open neighborhood $U_A\supset A$ such that for every $ p\in U_A\setminus A$ there exists a unique point $f_A(p)\in A$ such that $d(p,f_A(p))=d_A(p)$, and a  unique minimizing geodesic joining $p$ with $f_A(p)$. We denote by $\mathcal R(M)$ the class of compact sets of positive reach in $M$.
\end{defi}

By the previous definition, there is a well-defined map 
\begin{equation}\label{eq:descripcionFA}
 F_A : U_A \setminus A \hasta SM,\qquad   F_A(p) = \left(\gamma(0),\gamma'(0)\right)
\end{equation}
where $\gamma$ is the unique minimizing geodesic such that $\gamma(0) = f_A(p)$ and $\gamma(d_A(p)) = p$.

It was shown by Kleinjohann (\cite[Satz 3.3]{kleinjohann}) that $N(A) := F_A(U_A\setminus A)$ is a naturally oriented compact Lipschitz submanifold of $SM$. The corresponding current, also denoted by $N(A)$, is  called the {\em normal cycle} of $A$. It follows from Proposition \ref{prop:lipschitz2} below that $N(A)$ is legendrian (i.e. it vanishes on multiples of $\alpha$).

\begin{prop}[{{\cite[Satz 3.3, Korollar 2.7]{kleinjohann}}}]\label{prop:ciclonormalfundamental} Given a set of positive reach $A$ in $M$ there exists $r=r_A>0$ such that $A_r\subset U_A$ and

\begin{enumerate}[i)]
    \item for $0<t<r$ the restriction $\restr{F_A}{\partial A_t}$ gives a bilipschitz homeomorphism between $\partial A_t$ and $N(A)$, preserving the natural orientations,
    \item the distance function $d_A$ is of class $C^1$ in $A_{r}\setminus A$ and
    \[
     \phi_{d_A(p)}(F_A(p))=(p,\nabla d_A(p)),\qquad \partial A_t=d_A^{-1}(\{t\})
    \] for $0<t<r$.
    In particular, each level set $\partial A_t$ with $0<t<r$ is a $C^1$-regular hypersurface with unit normal vector field $\nabla d_A$.
\end{enumerate}
\end{prop}

The following propositions are certainly well-known.

\begin{prop}\label{prop:tubosetpositivereach}
    For $0<s<r=r_A$ the set $A_s$ has positive reach and on $A_r\setminus A_s$ we have
    \begin{equation}\label{eq:dist_tub}
         d_{A_s}=d_A-s,\qquad F_{A_s}=\phi_s\circ F_A.
    \end{equation}In particular $(A_s)_t=A_{t+s}$ for $t+s<r$.
\end{prop}

\begin{proof}
    Let $p \in A_r\setminus A_s$, and put $d=d_A(p)$. Let $\gamma\:[0,d]\to A_r$ be the unique minimizing geodesic with $\gamma(0)=f_A(p)$ and $\gamma(d)=p$. In particular $|\gamma'|=1$ and thus $\gamma(s)\in A_s$.
    
    Assume that $\restr{\gamma}{[s,d]}$ does not minimize the distance between $p$ and $A_s$, i.e., there exists a smooth curve $\alpha : [0,1] \to M$ with $q:=\alpha(0)\in A_s$, $\alpha(1)=p$ and length $\ell(\alpha) < d-s$. It follows that 
    \[
    d_A(p)\leq \ell(\alpha)+d_A(q)\leq\ell(\alpha)+s<d_A(p),
    \]
    a contradiction. We conclude that $\gamma|_{[s,d]}$ realizes the distance $d_{A_s}(p)$. Hence $d_{A_s}(p)=d_A(p)-s$ and 
     \[F_{A_s}(p)=(\gamma(s),\gamma'(s))=\phi_s(\gamma(0),\gamma'(0))=\phi_s(F_A(p)).\]

\end{proof}

\begin{prop}\label{prop:lipschitz1}
For $0<s<r_A$, the restriction $\restr{\phi_s}{N(A)}$ is a bilipschitz homeomorphism between $N(A)$ and $N(A_s)$.
\end{prop}

\begin{proof}

Take $t$ with $s<t<\min(r_A,s+r_{A_s})$. 
By Proposition \ref{prop:ciclonormalfundamental}, both $\restr{F_A}{\partial A_t}\: \partial A_t \to N(A)$ and $\restr{F_{A_s}}{\partial A_t} \: \partial A_t \to N(A_s)$  are bilipschitz homeomorphisms. By \eqref{eq:dist_tub}  we have
\[
\restr{\phi_s}{N(A)}=\restr{F_{A_s}}{\partial A_t}\circ(\restr{F_A}{\partial A_t})^{-1}.
\]
The statement follows.
\end{proof}

\begin{prop}\label{prop:lipschitz2}For $0<t<r_A$
the composition $\pi \circ \phi$ gives a bijective Lipschitz map between $N(A) \times (0,t]$ and $A_t \setminus A$.
\end{prop}
\begin{proof}
Since $\pi,\phi$ are smooth, the restriction of $\pi\circ\phi$ to the Lipschitz manifold $N(A)\times (0,t]$ is clearly Lipschitz. 

Given $(\xi,s)\in N(A)\times (0,t]$, we know by the previous proposition that $\phi(\xi,s)\in N(A_s)$ and thus $\pi\circ \phi(x,s)\in\partial A_s\subset A_t\setminus A$.

To check surjectivity, given $p\in A_t\setminus A$ take $\xi=F_A(p), s=d_A(p)$ and note that $\pi\circ\phi(\xi,s)=p$.    

As for injectivity, suppose $\pi\circ \phi(\xi_1,t_1)=\pi\circ \phi(\xi_2,t_2)=:p$ for some $(\xi_1,t_1),(\xi_2,t_2)\in N(A)\times (0,t]$. By the previous proposition $p$ belongs to both $\partial A_{t_1},\partial A_{t_2}$, so $t_1=t_2$. For $s\in[0,t_1]$, the geodesics $\gamma_1(s)=\pi \circ \phi(\xi_1,s),\gamma_2(s)=\pi \circ \phi(\xi_2,s)$ realize the distance between $p$ and $A$. Since $A_s\subset A_{r_A}\subset U_A$, we have $\gamma_1=\gamma_2$ and thus $\xi_1=\xi_2$.
\end{proof}

\begin{cor}\label{cor:igualdad_tubos}
For every $A\in \mathcal R(M)$ and $\mu\in \mathcal V(M)$ we have $\mu(A_t) = \mathbf{T}_t\mu(A)$ for $0\leq t\leq r_A$.
\end{cor}

\begin{proof} Let $\mu = \llbracket \omega,\eta\rrbracket$. By Propositions \ref{prop:lipschitz1} and \ref{prop:lipschitz2} and the coarea formula,
    \begin{align*}
    \mu(A_t) &= \int_{N(A_t)}\omega + \int_{A_t} \eta = \\
    &= \int_{\phi_t(N(A))} \omega + \int_{(\pi \circ \phi)(N(A)\times (0,t])} \eta +
     \int_A \eta = \mathbf{T}_t\mu(A).\qedhere
    \end{align*}
\end{proof}

\begin{ob}
    In the subclass $\mathcal P(M)\subset\mathcal R(M)$ of compact submanifolds with corners, the normal cycle is more naturally defined as follows. For $A \in \mathcal{P}(M)$ and $p \in A$, let
\begin{align*}T_p A&=\left\{\gamma'(0) \in T_p M : \gamma \in  C^1([0,1),A), \gamma(0) = p\right\}\\
N'(A) &= \{(p,v) \in SM : p \in A, \langle v, w\rangle\leq 0\ \forall w \in T_p A\}.
\end{align*}
Let us check that indeed $N'(A)$ equals $N(A)=F_A(U_A)$. Covering $A$ by local charts (locally modelled on $\R^k\times [0,\infty)^l\subset \R^m$), and considering the copy of $N'(A)$ in the cosphere bundle of $M$, one sees that $N'(A)$ is a compact topological manifold. 

It is also easy to show that $N(A)\subset N'(A)$. It then follows by the invariance of domain theorem that $N(A)$ is an open subset of $N'(A)$. Since $N'(A)$ is a Hausdorff space and $N(A)$ is compact, we also have that $N(A)$ is a closed subset of $N'(A)$. Since the number of connected components of both $N(A),N'(A)$ clearly equals the number of connected components of $A$, we necessarily have $N(A)=N'(A)$.
\end{ob}

\subsection{Derivative operators in $\RR{m}$ and $\CC{n}$}

Given $\lambda \in \R$ let $\derivada{\lambda,\R}\colon \vlr\to\vlr$ be the restriction of $\derivada{\RR{m}}$ to $\vlr$, and let $\derivada{\lambda,\C}$ be the restriction of $\derivada{\CC{m}}$ to $\vlc$.

\begin{prop}\label{prop:der_sigma}
\begin{align}
&\derivada{\lambda,\R} \sigma_i^{\lambda} = (m-i)\sigma_{i-1}^{\lambda}-\lambda (i+1)\sigma_{i+1}^{\lambda},& 0 \leq i \leq m-2,\label{eq:der_sigma_real}\\
&\derivada{\lambda,\R} \sigma_{m-1}^{\lambda} = \sigma_{m-2}^{\lambda},\label{eq:der_area}\\
&\derivada{\lambda,\R} \sigma_{m}^{\lambda} = \sigma_{m-1}^{\lambda}\label{eq:der_vol},
\end{align}where it is understood that $\sigma_{-1}^\lambda=0$.
\end{prop}
Let us emphasize that \eqref{eq:der_sigma_real} would make formal sense but does not hold for $i=m-1$.
\begin{proof}
By \cite[Lemma 3.1]{FuLeg}, putting $\kappa_m=0$, we have
\[
d\kappa_{i} = \alpha \wedge\left((m-i) \kappa_{i-1}-\lambda(i+1) \kappa_{i+1}\right), \quad 0 \leq i \leq m-1.
\]
Contracting with $T$ yields
\[
i_T d\kappa_i = (m-i)\kappa_{i-1} - \lambda (i+1)\kappa_{i+1}, \quad \mod(\alpha,d\alpha).
\]
By Proposition \ref{prop:derivada} the result follows. 
\end{proof}

\begin{lema} The following equalities hold modulo $\alpha,d\alpha$ 
\begin{enumerate}[i)]
        \item For $k > 2q$
	    \begin{equation}\label{eq:iTdbeta}
        \begin{split}
		\dfrac{\omega_{2n-k}}{\omega_{2n-k+1}}i_{T} d\beta_{k,q}\equiv &\, 2 (n-k+q+1)\gamma_{k-1,q}\\ &
		+{(k-2q+1)}\beta_{k-1,q-1} \\
		&-\frac{\lambda}{2\pi} {(k-2q+1)(2n-k+1)}\beta_{k+1,q}.
        \end{split}
	    \end{equation}

    	\item For $n > k-q$
		\begin{equation}\label{eq:iTdgamma}
        \begin{split}
		\dfrac{\omega_{2n-2q}}{\omega_{2n-2q+1}}i_T d\gamma_{2q,q} \equiv&\, \beta_{2q-1,q-1} \\
		&-\frac{\lambda}{2\pi} \frac{(q+2)(2n-2q+1)}{n-q}\beta_{2q+1,q} \\
		&-\frac{\lambda}{2\pi} \frac{(n-q-1)(2n-2q+1)}{n-q}\gamma_{2q+1,q}.
        \end{split}
        \end{equation}

\end{enumerate}
\end{lema}

\begin{proof}
    This is a straightforward computation using \cite[ Lemma 3.3, Lemma 3.6]{JuditGil}.
\end{proof}

\begin{prop}\label{prop:der_mu}
For $k > 2q$
\begin{align}
\dfrac{\omega_{2n-k}}{\omega_{2n-k+1}}\derivada{\lambda,\C}\mu_{k, q}^\lambda&=(k-2q+1) \mu_{k-1,q-1}^\lambda + 2 (n-k+q+1)\mu_{k-1,q}^\lambda \label{eq:der_kq}\\
&- \dfrac{\lambda}{2 \pi } (2 n-k+1)\left((k-2 q+1) \mu _{k+1,q}^\lambda +2 (q+1) \mu _{k+1,q+1}^\lambda \right)\notag
\end{align}
and
\begin{equation}\label{eq:der_2qq}
\dfrac{\omega_{2n-2q}}{\omega_{2n-2q+1}} \derivada{\lambda,\C} \mu_{2q,q}^\lambda = \mu_{2q-1,q-1} - (2n-2q+1) \dfrac{\lambda}{2 \pi }\mu_{2q+1,q}.
\end{equation}
\end{prop}

\begin{proof}
Equality \eqref{eq:der_kq} follows from Proposition \ref{prop:derivada}, using \eqref{eq:iTdbeta}  and the following (see \cite[Proposition 2.7]{JuditGil})
\begin{equation}\label{eq:gammaglob}
\llbracket \gamma_{k,q},0 \rrbracket = \mu_{k,q}^\lambda - \lambda \dfrac{(2n-k)(q+1)}{2\pi (n-k+q)}\mu_{k+2,q+1}^\lambda, \quad n-k+q > 0.
\end{equation}

Let us now prove \eqref{eq:der_2qq}. 
Note first that from \eqref{eq:iTdgamma} and \eqref{eq:gammaglob} we get

	\begin{equation}\label{eq:iTgammaglob}
	\begin{split}
	\llbracket i_T d\gamma_{2j,j},0\rrbracket &=\dfrac{\omega_{2n-2j+1}}{\omega_{2n-2j}} \mu_{2j-1,j-1}^\lambda \\
	&-\dfrac{\omega_{2n-2j+1}}{\omega_{2n-2j}} \dfrac{(2n-2j+1)(n+1)}{n-j}\dfrac{\lambda}{2\pi}\mu_{2j+1,j}^\lambda \\
	&+ \dfrac{\omega_{2n-2j+1}}{\omega_{2n-2j}}\dfrac{(2n-2j+1)(2n-2j-1)(j+1)}{n-j}\dfrac{\lambda^2}{4\pi^2}\mu_{2j+3,j+1}^\lambda\\
        &=: a_j \mu_{2j-1,j-1}^\lambda+b_j\frac{\lambda}{\pi}\mu_{2j+1,j}^\lambda+c_j\frac{\lambda^2}{\pi^2}\mu_{2j+3,j+1}^\lambda
	\end{split}
	\end{equation}

Then, by Proposition \ref{prop:derivada} and observing that $a_{n} = 2$

\begin{align*}
\derivada{\lambda,\C}\mu_{2q,q}^\lambda &= \sum_{i = 0}^{n-q-1}\left(\dfrac{\lambda}{\pi}\right)^i \dfrac{(q+i)!}{q!}\llbracket i_T d\gamma_{2q+2i,q+i},0\rrbracket + 2 \left(\dfrac{\lambda}{\pi}\right)^{n-q}\frac{n!}{q!}\mu_{2n-1,n-1}^\lambda\\
&= \sum_{i = 0}^{n-q}\left(\dfrac{\lambda}{\pi}\right)^i \dfrac{(q+i)!}{q!}a_{q+i}\mu_{2q+2i-1,q+i-1}^\lambda \\
    &+\sum_{i = 0}^{n-q-1}\left(\dfrac{\lambda}{\pi}\right)^{i+1} \dfrac{(q+i)!}{q!}b_{q+i}\mu_{2q+2i+1,q+i}^\lambda \\
    &+\sum_{i = 0}^{n-q-2}\left(\dfrac{\lambda}{\pi}\right)^{i+2} \dfrac{(q+i)!}{q!}c_{q+i}\mu_{2q+2i+3,q+i+1}^\lambda \\
    &= a_q\mu_{2q-1,q-1} + \frac{\lambda}{\pi}((q+1)a_{q+1} + b_q)\mu_{2q+1,q}^\lambda \\
    &+ \sum_{j = 2}^{n-q} \left(\dfrac{\lambda}{\pi}\right)^j\left(\dfrac{(q+j)!}{q!}a_{q+j} + \dfrac{(q+j-1)!}{q!}b_{q+j-1} +\dfrac{(q+j-2)!}{q!}c_{q+j-2}\right) \mu_{2q+2j-1,q+j-1}^\lambda
\end{align*}

A straightforward computation using $k\omega_k=2\pi\omega_{k-2}$ shows 
\begin{equation*}
    j(j-1)a_j+(j-1)b_{j-1}+c_{j-2}=0
\end{equation*}
and the result follows.
\end{proof}

Note that by \eqref{eq:def_Lambda_L_H} the linear map $\Phi_0\: \Val^{U(n)}\to\Val^{U(n)}$ given by $\restr{\Phi_0}{\Val_k^{U(n)}}=\omega_{2n-k}\nm{id}$ satisfies \[\partial_{0,\C}=\Phi_0\circ \Lambda \circ\Phi_0^{-1}.\]
Remarkably, a similar identity holds for all $\lambda$, which will be crucial for our determination of tube formulas in $\CC{n}$.
\begin{teo}\label{teo:formulasor}
The linear isomorphism
\[
\Phi_\lambda=\mathcal F_{\lambda,\C}\circ \Phi_0\:\nm{Val}^{U(n)} \hasta \mathcal{V}^n_{\lambda,\C}, \qquad \mu_{k,q}\longmapsto \omega_{2n-k}\mu_{k,q}^\lambda.
\]
fulfills
\[
\derivada{\lambda,\C} = \Phi_\lambda\circ\left(\Lambda - \lambda L\right)\circ\Phi_\lambda^{-1}.
\]
\end{teo}

\begin{proof}
By combining Proposition \ref{prop:der_mu}, Proposition \ref{propformulas} and the fact $\frac{\omega_n}{\omega_{n-2}} = \frac{2\pi}{n}$, this is straightforward to check:
\begin{align*}
    \Phi_\lambda\circ (\Lambda - \lambda L)(\mu_{k,q}) &= (k-2q+1)\omega_{2n-k+1}\mu_{k-1,q-1}^\lambda +2(n-k+q+1)\omega_{2n-k+1}\mu_{k-1,q}^\lambda \\
    &- \lambda(k-2q+1)\omega_{2n-k-1}\mu_{k+1,q}^\lambda - 2\lambda(q+1)\omega_{2n-k-1}\mu_{k+1,q+1}^\lambda \\
    &= \omega_{2n-k+1}\left((k-2q+1) \mu_{k-1,q-1}^\lambda + 2 (n-k+q+1)\mu_{k-1,q}^\lambda\right. \\
    &- \dfrac{\lambda}{2 \pi } (2 n-k+1)\left((k-2 q+1) \mu _{k+1,q}^\lambda +2 (q+1) \mu _{k+1,q+1}^\lambda \right) \\
    &= \omega_{2n-k}\partial_{\lambda,\C}\mu_{k,q}^\lambda = \partial_{\lambda,\C}\circ \Phi_\lambda(\mu_{k,q}).\qedhere
\end{align*}

\end{proof}

A similar phenomenon holds in real space forms, but restricted to a hyperplane of $\mathcal V_{\lambda,\R}^m$.
\begin{teo}\label{teo:formulasor_real}
The linear monomorphism
\[
\Psi_\lambda : \nm{Val}^{O(m)} \hasta \mathcal{V}^{m+1}_{\lambda,\R}, \qquad \mu_k \longmapsto \sigma_k^{\lambda}
\]
fulfills
\[
\derivada{\lambda,\R} \circ \Psi_\lambda = \Psi_\lambda\circ\left(\Lambda- \lambda L\right).
\]
\end{teo}
\begin{proof}
By Proposition \ref{prop:der_sigma} and Theorem \ref{teo:formulasplanas}
\[
\begin{aligned}
\derivada{\lambda,\R} \circ \Psi_\lambda(\mu_k) &= \derivada{\lambda,\R}\sigma_k^{\lambda} = (m-k+1)\sigma_{k-1}^\lambda - \lambda (k+1)\sigma_{k+1}^\lambda \\
&= \Psi_\lambda((m-k+1)\mu_{k-1} - \lambda (k+1)\mu_{k+1}) \\
&= \Psi_\lambda(\Lambda\mu_k - \lambda L\mu_{k}).
\end{aligned}
\]
\end{proof}

Note the difference of dimensions between the source and the target of $\Psi_\lambda$. We will show that there is no isomorphism between $\Val^{O(m)}$ and $\vlr$ intertwining $\partial$ and $\Lambda-\lambda L$. This is essentially due to the fact that \eqref{eq:der_area} and \eqref{eq:der_vol} differ from \eqref{eq:der_sigma_real}.

\section{A model space for tube formulas}

We next perform some abstract computations that will easily lead to the tube formulas in both complex and real space forms via \eqref{eq:conmutan} and \eqref{eq:conmutan_real}. The same approach will allow us to determine the kernel, the image, and the spectrum of the derivative operator $\partial$ on these spaces.

\subsection{A system of differential equations}\label{subsec:ode}

It is well-known that the operators $X = x\frac{\partial}{\partial y}$, $Y = y \frac{\partial}{\partial x}$ and $H=[X,Y]$ induce an $\mathfrak{sl}_2$-structure on $\C[x,y]$. The decomposition into irreducible components is $\C[x,y]=\bigoplus_{m\geq 0} V^{(m)}$ where $V^{(m)}$ is the subspace of $m$-homogeneous polynomials:
\[
V^{(m)} := \nm{span}_\C\displaystyle\{x^ky^{m-k}\}_{k = 0}^m.
\]
One has $H(x^ky^{m-k})=(m-2k) x^k y^{m-k}$.

Motivated by Theorem \ref{teo:formulasor}, we consider $Y_\lambda = Y - \lambda X$, {which is a derivation on $\C[x,y]$. It will be sometimes convenient to consider the monomials ${m\choose k}x^ky^{m-k}$. In these terms

\begin{align}
{m\choose k}Y_\lambda(x^ky^{m-k}) = (m-k+1){m\choose k-1}x^{k-1}y^{m-k+1} - \lambda(k+1){m\choose k+1} x^{k+1}y^{m-k-1}.\label{eq:Ylambda_monomio}
\end{align}
}

Our goal here is to solve the following Cauchy problem: find $p_k\:\R\to V^{(m)}$ such that
\begin{equation}\label{cauchyproblem}
 \partial_t p_k(t)=Y_\lambda p_k(t) , \quad p_k(0) = \binom{m}{k}x^{k}y^{m-k}, \quad 0 \leq k \leq m,
\end{equation}
i.e. to compute\begin{equation}\label{solucioncauchy}
p_k(t) = \binom{m}{k}\exp(tY_\lambda)(x^{k}y^{m-k}), \quad 0 \leq k \leq m.
\end{equation} 

We will use the standard notation 

\begin{equation}\label{eq:sin}
\operatorname{sin}_{\lambda}(t):=\left\{\begin{array}{cc}
\dfrac{\sin (\sqrt{\lambda} t)}{\sqrt{\lambda}} & \lambda>0, \\
\\
t & \lambda=0, \\
\\
 \dfrac{\sinh (\sqrt{|\lambda|} t)}{\sqrt{|\lambda|}} & \lambda<0,
\end{array}\right.
\end{equation}
which is an analytic function in both $\lambda$ and $t$, and $\cos_\lambda(t) := \frac{d}{dt}\sin_\lambda(t)$.
\begin{prop}
For any $\lambda,t \in \R$, we have
\[
\exp(tY_\lambda)x = x \cos_\lambda(t) + y \sin_\lambda(t){=:u}, \quad \exp(tY_\lambda)y = y \cos_\lambda(t)-\lambda x \sin_\lambda(t){=:v}.
\]
\end{prop}

\begin{proof}
Since clearly
\[
Y_\lambda^{2k}x = (-\lambda)^k x, \quad Y_\lambda^{2k+1}x = (-\lambda)^k y,
\]
we have
\[
\begin{aligned}\exp(tY_\lambda)x&=\sum_{k \geq 0}\dfrac{t^k}{k!}Y_{\lambda}^{k}x\\
&=\sum_{k \geq 0} \dfrac{t^{2k}}{(2k)!}(-\lambda)^k x+ \sum_{k \geq 0}\dfrac{t^{2k+1}}{(2k+1)!}(-\lambda)^k y  \\
&= x\cos_\lambda(t) + y \sin_\lambda(t).
\end{aligned}
\]
In the same way we can compute $\exp(tY_\lambda)y$.
\end{proof}

The following standard and elementary fact will be useful.

\begin{lema}
Let $\mathbf{A}$ be an algebra. A vector field on $\mathbf A$ is a derivation iff its flow $\phi_t$ satisfies
\[
\phi_t(pq) = \phi_t(p)\phi_t(q), \quad \forall p,q \in \mathbf{A},\forall t\in\R.
\]
In other words, each $\phi_t$ is an $\mathbf{A}$-morphism.
\end{lema}

\begin{teo}
The solution of the Cauchy problem \eqref{cauchyproblem} is 
\begin{align}
p_k(t) &=  \binom{m}{k}u^k v^{m-k}\\
\label{eq:pk_factores}&=\binom{m}{k}(x\cos_\lambda(t) + y \sin_\lambda(t))^k (y \cos_\lambda(t) - \lambda x \sin_\lambda(t))^{m-k}\\
\label{eq:pk_desarrollo}&= {\sum_{j = 0}^{m}\phi_{m,k,j}^{\lambda}(t)\binom{m}{j}x^j y^{m-j}},
\end{align}
 where
\begin{equation}\label{eq:poliphi}
\phi_{m,k,j}^\lambda(t) = \sum_{h \geq 0}(-\lambda)^{j-h}\binom{m-j}{k-h}\binom{j}{h}\sin_\lambda^{k+j-2h}(t)\cos_\lambda^{m-k-j+2h}(t).
\end{equation}
\end{teo}

\begin{proof}
Since $Y_\lambda$ is a derivation, $\exp(tY_\lambda)$ is a $\C[x,y]$-morphism by the previous lemma. Hence
\[
\exp(tY_\lambda)x^{k}y^{m-k} = (\exp(tY_\lambda) x)^k (\exp(tY_\lambda) y)^{m-k} = u^k v^{m-k}.
\]
Comparing with \eqref{solucioncauchy} yields \eqref{eq:pk_factores}. 

It remains to prove \eqref{eq:pk_desarrollo}. Putting $s=\sin_\lambda(t),c=\cos_\lambda(t)$ we have

\begin{align*}
{m\choose k}(xc + y s)^k(yc - \lambda x s)^{m-k}&= {m\choose k} \sum_{a,b}{k\choose a}(ys)^a (xc)^{k-a}{m-k\choose b}(-\lambda xs)^b(yc)^{m-k-b}\\
 &= {m\choose k} \sum_{a,b}{k\choose a}{m-k\choose b}(-\lambda )^b s^{a+b} c^{m-a-b} x^{k-a+b} y^{m-k+a-b}\\
 &={m\choose k}\sum_{j,h} {k\choose h} {m-k\choose j-h} (-\lambda)^{j-h} s^{j+k-2h}c^{m-j-k+2h} x^j y^{m-j}
\end{align*} where we changed $a=k-h,b=j-h$. Using
\begin{equation}\label{eq:binomiales}
{m\choose k}{k\choose h}{m-k\choose j-h}={m-j\choose k-h}{j\choose h}{m\choose j}
\end{equation}
yields \eqref{eq:pk_desarrollo}.
\end{proof}

\subsection{Eigenvalues and eigenvectors of $Y_\lambda$}

Given $f\: V \to V$ an endomorphism of $\C$-vector spaces, we denote by $\nm{spec}(f)$ the set of eigenvalues of $f$ and by $E_\alpha(f)$ the eigenspace associated to each $\alpha \in \nm{spec}(f)$.

\begin{lema}\label{lemaauto}
The endomorphism $\restr{Y_\lambda}{V^{(m)}}$ is diagonalizable with simple multiplicities and
\[
\nm{spec}(\restr{Y_\lambda}{V^{(m)}}) = \suce{(2k-m)\sqrt{-\lambda} : 0 \leq k \leq m},
\]
\[
E_{(2k-m)\sqrt{-\lambda}}(\restr{Y_\lambda}{V^{(m)}}) = \nm{span} \{e_1^ke_2^{m-k}\},
\]
where $e_1 := \sqrt{-\lambda}x + y$ and $e_2:= -\sqrt{-\lambda}x + y$.
\end{lema}

\begin{proof} The result is trivial to check for $m=1$ as
\[
Y_\lambda(e_1)=\sqrt{-\lambda}y+\lambda x=\sqrt{-\lambda}e_1,\qquad Y_\lambda(e_2)=-\sqrt{-\lambda}y-\lambda x=-\sqrt{-\lambda}e_2.
\]

Since $Y_\lambda$ is a derivation
\[
\begin{aligned}
&Y_\lambda e_1^{k} = k e_{1}^{k-1}Y_\lambda e_1 = k \sqrt{-\lambda} e_1^{k},\\ 
&Y_\lambda e_2^{m-k} = (m-k) e_2^{m-k-1}Y_\lambda e_2 = -\sqrt{-\lambda}(m-k) e_2^{m-k}.
\end{aligned}
\]
Hence
\[
Y_\lambda(e_1^{k}e_2^{m-k}) = (2k-m) \sqrt{-\lambda} e_1^{k}e_2^{m-k},
\]
as stated.
\end{proof}

\begin{ob}
It is interesting to notice that the spectra of $Y_\lambda$ and $\sqrt{-\lambda} H$, when restricted to each $V^{(m)}$, are identical. These two operators are thus intertwined e.g. by the linear isomorphism $x^ky^{m-k}\mapsto e_1^{k}e_2^{m-k}$.
\end{ob}

\subsection{Image of $Y_\lambda$}

Using Lemma \ref{lemaauto}, we can conclude that  $\restr{Y_\lambda}{V^{(m)}}$ is bijective if and only if $m$ is odd. If $m$ is even, then the kernel is one-dimensional. An explicit description is the following.

\begin{prop}\label{prop:ecuacionkerYlambda}
    If $m$ is even, then
    \begin{equation}\label{eq:im_ker}
        \nm{im}(\restr{Y_\lambda}{V^{(m)}}) = \ker Z_{m,\lambda},
    \end{equation}
    where
    \[
    Z_{m,\lambda} := \left(\dfrac{\partial^{2}}{\partial x^{2}} + \lambda \dfrac{\partial^{2}}{\partial y^{2}}\right)^{m/2}.    
    \]
\end{prop}

\begin{proof}
   By the binomial formula
    \begin{align}\notag
    Z_{m,\lambda}(x^k y^{m-k})& = \sum_{i = 0}^{m/2} \lambda^{m/2- i}\binom{m/2}{i}\dfrac{\partial^{m}}{\partial x^{2i}\partial y^{m-2i}}x^k y^{m-k}\delta_{k,2i}\\ &=\lambda^{\frac{m-k}2} \binom{m/2}{k/2}k!(m-k)!\label{eq:Z_monomio}
    \end{align}
    if $k$ is even, and $Z_{m,\lambda}(x^k y^{m-k}) = 0$ if $k$ is odd. Therefore
    \begin{align*}
            Z_{m,\lambda}\circ Y_\lambda (x^{2l+1}y^{m-2l-1}) =&\, Z_{m,\lambda}((2l+1) x^{2l}y^{m-2l} - \lambda (m-2l-1)x^{2l+2}y^{m-2l-2}) \\
        =&\, \lambda^{\frac{m}{2}-l} \binom{m/2}{l} (2 l+1)! (m-2l)!\\
        & -\lambda^{\frac{m}{2}-l}  \binom{m/2}{l+1} (2l+2)! (m-2l-1)! = 0\\
        Z_{m,\lambda}\circ Y_\lambda(x^{2l}y^{m-2l}) =&\, 0.
    \end{align*}
   This shows that $\text{im}(Y_\lambda)$ is a subspace of $\ker Z_{m,\lambda}$. Given that $Z_{m,\lambda}$ is not zero, we have $\text{dim} \ker Z_{m,\lambda} = m$, and by Lemma \ref{lemaauto}, we know that the image of $\restr{Y_\lambda}{V^{(m)}}$ has the same dimension. This yields \eqref{eq:im_ker}
\end{proof}

Next we compute, for even $m$ and given $\varphi$ in the image of $\restr{Y_\lambda}{V^{(m)}}$, the preimage $Y_\lambda^{-1}(\{\varphi\})$.

Consider
\begin{equation}\label{def:Pk}
    P_{m,k} := \sum_{j \geq 0} \lambda^j \dfrac{(k+2j-1)!! (m-k-2j-1)!!}{(k-1)!!(m-k+1)!!}\binom{m}{k+2j}x^{k+2j}y^{m-k-2j} \in V^{(m)}.
\end{equation}

A simple computation using \eqref{eq:Ylambda_monomio} shows
 \[
Y_\lambda P_{m,k}=\binom{m}{k-1}x^{k-1}y^{m-k+1}-c_{m,k} x^m
\]
where $c_{m,k}=0$ if $m-k$ is even, and otherwise
\[
c_{m,k}=\lambda^{\frac{m-k+1}2}\frac{m!!}{(k-1)!!(m-k+1)!!}.
\]

With these ingredients at hand, for even $m$, we can now compute a preimage by $Y_\lambda$ of any element in $\nm{im} Y_\lambda$ as follows.
\begin{prop}\label{prop:formulapreimagenYlambda}
    Let $\Pi \: V^{(m)} \to V^{(m)}$ be given by $\binom{m}{k}x^k y^{m-k} \mapsto P_{m,k+1}$. If $m$ is even then 
    \begin{equation}\label{eq:formulapreimagen}
        {Y_\lambda \circ \Pi }(\varphi) =\varphi,\qquad \forall \varphi\in \nm{im}\restr{Y_\lambda}{V^{(m)}}
    \end{equation}
\end{prop}
\begin{proof}

Let $0< k < m$. Since $(m-k+1)c_{m,k} - \lambda (k+1)c_{m,k+2} = 0$, using \eqref{eq:Ylambda_monomio} we get
\begin{align*}
    Y_\lambda \circ \Pi \circ Y_\lambda \binom{m}{k}x^k y^{m-k} &= Y_\lambda \circ \Pi \left((m-k+1)\binom{m}{k-1}x^{k-1}y^{m-k+1} - \lambda (k+1)\binom{m}{k+1}x^{k+1}y^{m-k-1}\right) \\
    &=(m-k+1)Y_\lambda P_{m,k} - \lambda(k+1)Y_\lambda P_{m,k+2} \\
    &= (m-k+1)\binom{m}{k-1}x^{k-1}y^{m-k+1} - \lambda (k+1)\binom{m}{k+1}x^{k+1}y^{m-k-1}\\
    &+((m-k+1)c_{m,k}-\lambda (k+1)c_{m,k+2})x^m \\
    &= Y_\lambda\binom{m}{k}x^k y^{m-k}.
\end{align*}
For $k = 0$ and $k = m$,
\begin{align*}
Y_\lambda \circ \Pi \circ Y_\lambda (y^m) &=  -\lambda Y_\lambda(P_{m,2})=-\lambda(mx y^{m-1} -  c_{m,2}x^m)=Y_\lambda(y^m) +\lambda c_{m,2}x^m, \\
Y_\lambda \circ \Pi \circ Y_\lambda (x^m) &= Y_\lambda(P_{m,m})=m x^{m-1}y - c_{m,m}x^m =Y_\lambda(x^m) + c_{m,m}x^m. 
\end{align*}
Since $c_{m,m} = 0$, and $c_{m,2} = 0$ if $m$ is even, the result follows.
\end{proof}

\section{Tube formulas in $\RR{m}$ and $\CC{n}$}
Here we will obtain our main result: the tube formulas for invariant valuations of $\CC{n}$ (i.e. the tubular operator $\mathbf{T}_t$ on $\vlc$). We will also recover Santaló's tube formulas for $\vlr$ (cf. \cite{santalopol}) in a way that explains the similarities between the real and the complex space forms.

\subsection{Tube formulas in complex space forms}

Recalling \eqref{eq:sl2_decomposition} and Proposition \ref{prop:primitives}, we get an isomorphism $I\colon W_n\to \Val^{U(n)}$ of $\mathfrak{sl}_2$-modules from 
\[
W_{n}:=\bigoplus_{0 \leq 2r \leq n} V^{(2n-4r)}
\] 
to $\Val^{U(n)}$ by putting $I(y^{2n-4r})=\pi_{2r,r}$ (i.e. mapping $Y$-primitive elements to $\Lambda$-primitive elements) and
\[
\begin{aligned}
{2n-4r\choose k-2r}I(x^{k-2r} y^{2n-k-2r})&=\frac1{(k-2r)! }I(X^{k-2r}(y^{2n-4r}))\\
&= \frac1{(k-2r)! }L^{k-2r}I(y^{2n-4r})=\frac1{(k-2r)! }\pi_{k,r}.
\end{aligned}
\] 

By Theorem \ref{teo:formulasor}, the map $J_{\lambda,\C}:=\Phi_\lambda\circ I\colon W_{n}\to \mathcal V_{\lambda,\C}^n$ fulfills
\begin{equation}\label{eq:conmutan}
 \derivada{\lambda,\C} \circ J_{\lambda,\C} =J_{\lambda,\C} \circ Y_\lambda.
\end{equation}

We define 
\begin{equation}\label{eq:definicionsigma}\sigma_{k,r}^\lambda := {2n-4r\choose k-2r}J_{\lambda}(x^{k-2r} y^{2n-k-2r})= \dfrac{\omega_{2n-k}}{(k-2r)!}\pi_{k,r}^\lambda
\end{equation}
and arrive at our main theorem.\begin{teo}\label{teo:main}
The tubular operator $\mathbf{T}_t$ in $\vlc$ is given by
\[
\mathbf{T}_t(\sigma_{k,r}^{\lambda}) =  \sum_{j = 0}^{2n-4r}\phi_{2n-4r,k-2r,j}^\lambda(t)\sigma_{j+2r,r}^{\lambda},
\]
where 
\[
\phi_{m,i,j}^\lambda(t) = \sum_{h \geq 0}(-\lambda)^{j-h}\binom{m-j}{i-h}\binom{j}{h}\sin_\lambda^{i+j-2h}(t)\cos_\lambda^{m-i-j+2h}(t).
\]
\end{teo}
\begin{proof}
By \eqref{eq:exponential_T}, using \eqref{eq:conmutan} and \eqref{eq:definicionsigma}, and putting $m=2n-4r$,  we get
\begin{align*}
    \mathbf T_t \sigma_{k,q}^\lambda&=\mathrm{exp}(t \derivada{\lambda,\C})(\sigma_{k,q}^\lambda)\\
    &={m\choose k-2r}\mathrm{exp}(t \derivada{\lambda,\C})\circ  J_{\lambda,\C}(x^{k-2r} y^{m-k+2r})\\
    &={m\choose k-2r}J_{\lambda,\C}\circ  \mathrm{exp} (t Y_\lambda)(x^{k-2r} y^{m-k+2r})\\
    &= J_{\lambda,\C} (p_{k-2r}(t)).
\end{align*}
Using \eqref{eq:pk_desarrollo} the result follows.
\end{proof}

The tube formulas in terms of the $\tau_{k,i}^\lambda$ can be obtained from Theorem \ref{teo:main} using \eqref{eq:pikrtasakis} and \eqref{eq:descLef} which hold verbatim replacing $\pi^\lambda_{k,r},\tau_{k,r}^\lambda$ for $\pi_{k,r},\tau_{k,r}$.

\begin{ob}
    The tube formula for the volume $\sigma_{2n,0}^{\lambda} = \nm{vol}_{\CC{n}}$ is given by the following simple expression
\[
\nm{vol}_{\CC{n}}(A_t)= \sum_{j = 0}^{2n}\sin_{\lambda}^{2n-j}(t)\cos_\lambda^{j}(t)\sigma_{j,0}^{\lambda}(A),
\]
which is Theorem 4.3 of \cite{bfs}, since $\sigma_{j,0}^\lambda = \omega_{2n-j}\tau_{j,0}^{\lambda} = \Phi_\lambda(\mu_j)$. The tube formulas $\mathbf{T}_t \sigma_{2n-2r,r}$ are equally simple.
\end{ob}

\begin{ob} 

An interesting feature of the previous tube formulas is the following self-similarity property, which is explained by \eqref{eq:conmutan}. Let 
\[
\mathbf{G}^{n,j}_{\lambda}\:\vlc\hasta\V^{n+2j}_{\lambda,\C}  ,\qquad \mathbf G^{n,j}(\sigma_{k,r}^{\lambda})=\sigma_{k+2j,r+j}^\lambda. 
\] 
Then one has $\mathbf T_t\circ \mathbf G^{n,j}=\mathbf G^{n,j}\circ \mathbf T_t$.
\end{ob}

\begin{ob}
    It is also worth noting that $\vlc=\bigoplus_{0 \leq 2r \leq n}\Inv{n}{r}$ where 
    \[
    \Inv{n}{r} := J_{\lambda,\C}(V^{(2n-4r)}) = \suce{\sigma_{k,r}^\lambda : 2r \leq k \leq 2n-2r},
    \]
    and that these subspaces are $\partial_{\lambda,\C}$-invariant. In particular, given $\varphi\in \Inv{n}{r}$ one has $\mathbf T_t(\varphi)\in\Inv{n}{r}$.
    
\end{ob}

\subsection{Tube formulas in real space forms}\label{tubularreal}

Let $I\: V^{(m)}\to \Val^{O(m)}$ be the isomorphism of irreducible $\mathfrak{sl}_2$-representations determined by $I(y^{m})=\chi$; i.e.
\begin{align*}
{m \choose i  }I(x^i y^{m-i})&= \frac1{i!} I(X^i(y^m))=\frac1{i!}L^i(I(y^m))\\
&=\frac1{i!}L^i(\mu_0)=\mu_i
\end{align*}where we used \eqref{eq:Lmu}. By Theorem \ref{teo:formulasor_real}, the map $J_{\lambda,\R}=\Psi_\lambda\circ I$ satisfies
\begin{equation}\label{eq:conmutan_real}
\derivada{\lambda,\R} \circ J_{\lambda,\R} = J_{\lambda,\R} \circ Y_\lambda.
\end{equation}

The map $J_{\lambda,\R}$ is explicitly given by
\begin{equation}\label{eq:J_real}
 J_{\lambda,\R} \: V^{(m)} \hasta \mathcal V_{\lambda,\R}^{m+1}, \qquad \binom{m}{i}x^{i}y^{m-i} \longmapsto \sigma_i^{\lambda}.
\end{equation}
The image of $J_{\lambda,\R}$ is the hyperplane $\mathcal H_\lambda^{m+1}:=\nm{im} J_{\lambda,\R}= \nm{span}\suce{\sigma_{0}^\lambda,\dots,\sigma_{m}^\lambda}$.

\begin{teo} The tubular operator on $\mathcal V_{\lambda,\R}^{m+1}$ is given as follows. For $ i = 0,\dots,m$,
\begin{equation}\label{eq:tubo_sigma_real}
\mathbf T_t\sigma_i^{\lambda} = \sum_{j = 0}^{m}\phi^\lambda_{m,i,j}(t)\sigma^\lambda_{j}.
\end{equation}
In particular
\begin{equation}\label{eq:tubo_area_real}
\mathbf T_t\sigma_{m}^\lambda  = \sum_{j = 0}^{m}\sin_\lambda^{m-j}(t)\cos_\lambda^{j}(t)\sigma^\lambda_{j},
\end{equation}
and thus
\begin{equation}\label{eq:tubo_vol_real}
\mathbf T_t\sigma_{m+1}^\lambda = \sum_{j = 0}^{m}\left(\int_{0}^{t}\sin_\lambda^{m-j}(s)\cos_\lambda^{j}(s)ds\right)\sigma^\lambda_{j} + \sigma_{m+1}^\lambda.
\end{equation}
\end{teo}
These formulas where first obtained by Santaló \cite{santalopol}.
\begin{proof} 
By \eqref{eq:exponential_T},\eqref{eq:conmutan_real} and \eqref{eq:J_real}, we have for $0 \leq i \leq m$,
\begin{align*}
    \mathbf T_t \sigma_i^{\lambda}&=\mathrm{exp}(t\derivada{\lambda,\C})(\sigma_i^{\lambda})\\
    &={m\choose i}\mathrm{exp}(t \derivada{\lambda,\C})\circ  J_{\lambda,\C}(x^{i} y^{m-i})\\
    &={m\choose i}J_{\lambda,\C}\circ  \mathrm{exp} (t Y_\lambda)(x^{i} y^{m-i})\\
    &=J_{\lambda,\C} (p_{i}(t)).
\end{align*}This proves \eqref{eq:tubo_sigma_real} of which \eqref{eq:tubo_area_real} is a particular case. Integrating with respect to $t$ yields \eqref{eq:tubo_vol_real}.
\end{proof}

\begin{ob} It is worth pointing out the similarity between tube formulas in real and complex space forms. More precisely, note that the isomorphism 
\begin{equation}\label{eq:Fnr}
\mathbf{F}_{n,r}\: \mathcal H_{\lambda}^{2n-4r+1} \hasta \Inv{n}{r}, \qquad   \sigma_{j}^\lambda\longmapsto \sigma_{k+2j,r+j}^\lambda
\end{equation} between the subspaces  $\mathcal H_{\lambda}^{2n-4r+1}\subset\mathcal{V}^{2n-4r+1}_{\lambda,\R} $  and $\Inv{n}{r}\subset \vlc$
commutes with the tubular operator $\mathbf{T}_t$. This is explained by \eqref{eq:conmutan} and \eqref{eq:conmutan_real}.
\end{ob}

\subsection{Spectral analysis of the derivative map}
Here we compute the eigenvalues and eigenvectors of $\derivada{\lambda,\R}$ and $\derivada{\lambda,\C}$. Note that the tube formulas for such valuations are extremely simple: if $\derivada{} \mu=a\mu$ with $a\in\C$, then $\mathbf{T}_t\mu=e^{at}\mu$.

\begin{prop}\label{prop:auto_complejo}
For $0 \leq 2r \leq n$, the restriction of $\derivada{\lambda,\C}$ to ${\Inv{n}{r}}$ has the following (simple) eigenvalues and eigenspaces:
\begin{align*}
\nm{spec}\left(\restr{\derivada{\lambda,\C}}{\Inv{n}{r}}\right) &= \suce{0,\pm 2\sqrt{-\lambda},\pm 4\sqrt{-\lambda},\dots,\pm 2(n-2r)\sqrt{-\lambda},},\\
{E_{(2k-2n+4r)\sqrt{-\lambda}}}&=\nm{span}_{\C}\suce{ J_{\lambda,\C}(e_1^{k}e_{2}^{2n-4r-k})},\qquad 0 \leq k \leq 2n-4r.
\end{align*}
Hence $\derivada{\lambda,\C}$ diagonalizes on $\mathcal V_{\lambda,\C}^n$ with the following eigenspaces:
\begin{align*}
E_{2j\sqrt{-\lambda}}(\derivada{\lambda,\C}) &= \nm{span}_\C \suce{J_{\lambda,\C} (e_1^{j+n-2r}e_2^{n-2r-j}) : 0 \leq 2r \leq {\min\{n-j,n+j\}}},
\end{align*}for $-n\leq j\leq n$.
\end{prop}
\begin{proof}
    Everything follows from Lemma \ref{lemaauto} and \eqref{eq:conmutan}.
\end{proof}

\begin{prop}\label{prop:auto_real}
\begin{enumerate}[i)]
 \item In $\RR{2n}$ the derivative operator is diagonalizable with
    \begin{align}
\nm{spec}(\derivada{\lambda,\R})& = \suce{0,\pm \sqrt{-\lambda},{\pm3\sqrt{-\lambda}},\dots, \pm (2n-1)\sqrt{-\lambda}},\\
 E_0(\derivada{\lambda,\R}) &= \nm{span}_\C\{\chi\}\\
    E_{(2k-2n+1)\sqrt{-\lambda}}(\derivada{\lambda,\R}) &= \nm{span}_\C \{J_{\lambda,\R}(e_1^{k}e_2^{2n-k-1})\}, \quad 0 \leq k \leq 2n-1\label{eq:eigen_par}
    \end{align}

    \item In $\RR{2n+1}$ the derivative operator is not diagonalizable since
    \begin{align}
\nm{spec}(\derivada{\lambda,\R}) &= \suce{0,0,\pm 2\sqrt{-\lambda},{\pm4\sqrt{-\lambda}},\dots,\pm 2n\sqrt{-\lambda}},\\
    E_0(\derivada{\lambda,\R}) &= \nm{span}_\C\{\chi\}\\
    E_{(2k-2n)\sqrt{-\lambda}}(\derivada{\lambda,\R}) &= \nm{span}_\C \{J_{\lambda,\R}(e_1^{k}e_2^{2n-k})\}, \quad 0 \leq k \leq 2n.\label{eq:eigen_impar}
    \end{align}
  \end{enumerate} 
\end{prop}

\begin{proof}
\begin{enumerate}[$i)$]
\item  By Lemma \ref{lemaauto} and \eqref{eq:conmutan_real} we have that $(2k-2n+1)\sqrt{-\lambda}$, $0 \leq k \leq 2n-1$, is an eigenvalue of ${\partial_{\lambda,\R}}$ with eigenspace given by \eqref{eq:eigen_par}. The Euler characteristic is clearly an eigenvector with zero eigenvalue. We thus have at least $2n+1$ eigenvalues. Since this is precisely the dimension of $\mathcal V_{\lambda,\R}^{2n}$, the statement follows.

\item In light of Lemma \ref{lemaauto} and \eqref{eq:conmutan_real}, we ascertain that $(2k-2n)\sqrt{-\lambda}$, $0 \leq k \leq 2n$, is an eigenvalue of ${\partial_{\lambda,\R}}$ and the corresponding eigenspace is described by \eqref{eq:eigen_impar}.

Our next objective is to prove that while the algebraic multiplicity of the zero eigenvalue is two, its geometric multiplicity is only one. This will entail finding a valuation $\mu$ that satisfies $\partial_{\lambda,\R}^2 \mu = 0$, while also ensuring that $\partial_{\lambda,\R} \mu \neq 0$. Consider $\sigma_{2n}^\lambda=J_{\lambda,\R}(x^{2n}) \in \mathcal{V}^{2n+1}_{\lambda,\R}$.  In the notation of Lemma \ref{lemaauto}, 
\[
x = \frac{1}{2\sqrt{-\lambda}}(e_1 - e_2), \qquad x^{2n} = (-4\lambda)^{-n}\sum_{i = 0}^{2n}(-1)^i \binom{2n}{i}e_1^{i}e_{2}^{2n-i}.
\]
Hence
\[
\partial_{\lambda,\R}\sigma_{2n+1}^\lambda = \sigma_{2n}^\lambda = (-4\lambda)^{-n}\sum_{i = 0}^{2n}(-1)^i \binom{2n}{i}J_{\lambda,\R}(e_1^{i}e_{2}^{2n-i}).
\]
Consider
\[
\nu := (-4\lambda)^{-n} \sum^{2n}_{\substack{i = 0\\
i\neq n}}\binom{2n}{i}\dfrac{(-1)^i}{(2i-2n)\sqrt{-\lambda}} J_{\lambda,\R}(e_1^{i}e_{2}^{2n-i}),
\]
and note that, by Lemma \ref{lemaauto},
\[
\partial_{\lambda,\R}\nu := (-4\lambda)^{-n} \sum^{2n}_{\substack{i = 0\\
i\neq n}}\binom{2n}{i}(-1)^{i} J_{\lambda,\R}(e_1^{i}e_{2}^{2n-i}),
\]
since  $e_1^{n}e_2^n \in \ker Y_\lambda$.
Finally, we define $\mu =  \sigma_{2n+1}^\lambda-\nu$. Then
\[
\partial_{\lambda,\R} \mu = (-4\lambda)^{-n}\binom{2n}{n}(-1)^n J_{\lambda,\R}(e_1^{n}e_2^{n})\neq 0,
\]
while
\begin{align*}
\partial_{\lambda,\R}^2 \mu &= (-4\lambda)^{-n}\binom{2n}{n}(-1)^n \partial_{\lambda,\R} J_{\lambda,\R}(e_1^{n}e_2^{n}) \\
&= (-4\lambda)^{-n}\binom{2n}{n}(-1)^n J_{\lambda,\R}(Y_{\lambda}(e_1^{n}e_2^{n})) = 0.
\end{align*}{It follows that $\dim\ker\partial_{\lambda,\R}<\dim\ker\partial_{\lambda,\R}^2$. Noting that $\chi\in\ker\partial_{\lambda,\R}$ this implies the statement.}\qedhere
\end{enumerate}
\end{proof}

\begin{ob}
    We conclude from Prosposition \ref{prop:auto_real} and  Lemma \ref{lemaauto} that there is no isomorphism between $\Val^{O(m)}$ and $\vlr$ intertwining $\Lambda- \lambda L$ and $\derivada{\lambda,\R}$. Indeed, these two operators have different spectra no matter the parity of $m$.
\end{ob}

\subsection{Stable valuations in complex space forms}

 We say that a valuation $\varphi\in \mathcal V(M)$ on a riemannian manifold $M$ is {\em stable} if $\partial \mu=0$, or equivalently, if $\mathbf{T}_t\mu=\mu$ for all $t$. By Proposition \ref{prop:auto_real} and \ref{prop:auto_complejo}, up to multiplicative constants, the Euler characteristic is the unique isometry-invariant stable valuation in $\RR{m}$. The complex case is more interesting.

\begin{prop}
The unique (up to multiplicative constants) stable valuation on $\mathcal I_\lambda^{n,r}$ is given by
\[
\psi_{2r} = \sum_{i = r}^{n-r}\binom{n-2r}{i-r}\binom{2n-4r}{2i-2r}^{-1}\lambda^{i-r} \sigma_{2i,r}^\lambda.
\]
\end{prop}
\begin{proof}By Lemma \ref{lemaauto} the kernel of $Y_\lambda$ on the space $V^{(m)}$ of homogeneous polynomials of degree $m=2n-4r$ is spanned by
    \[
\begin{aligned}
e_1^{n-2r}e_2^{n-2r} &= (y + \sqrt{-\lambda} x)^{n-2r}(y - \sqrt{-\lambda} x)^{n-2r} \\
&= (y^2 + \lambda x^2)^{n-2r} = \sum_{j = 0}^{n-2r}\binom{n-2r}{j}\lambda^j x^{2j} y^{m-2j} \\
&= \sum_{i = r}^{n-r}\binom{n-2r}{i-r}\binom{2n-4r}{2i-2r}^{-1}\lambda^{i-r} \binom{2n-4r}{2i-2r}x^{2i-2r} y^{2n-2i-2r}
\end{aligned}
\]Therefore the kernel of $\derivada{\lambda,\C}$ in  $\Inv{n}{r}$ is spanned by $\psi_{2r}=J_\lambda(e_1^{n-2r}e_2^{n-2r})$ , for each $0 \leq 2r \leq n$.
\end{proof}

Next we express the Euler characteristic as a combination of the stable valuations $\psi_{2r}$. Note in particular that $\chi$ is \emph{not} confined to any $\partial$-invariant subspace $\Inv{n}{r}$.

\begin{prop}
\[
\begin{aligned}
\chi &= \sum_{0 \leq 2r \leq n} \left(\frac{\lambda}{4\pi}\right)^r\binom{2r}{r}\dfrac{r!}{\omega_{2n-2r}}\psi_{2r}.
\end{aligned}
\]
\end{prop}

\begin{proof}
Since $\chi$ is stable, it can be expressed as $\chi =\sum_j a_j \psi_{2j}$. By  \cite[Theorem 3.11]{bfs}
\[
\begin{aligned}
\chi &=\left. \sum_{k,p \geq 0}\left(\frac{\lambda}{\pi}\right)^{k+p} \frac{\partial^{k+p}}{\partial \xi^{k} \partial \eta^{p}}\frac{1}{\sqrt{1-\xi}\sqrt{1-\eta}}\right|_{(0,0)} \tau_{2 k+2 p, p}^{\lambda}.
\end{aligned}
\]
The coefficient of $\tau_{2r,r}^{\lambda}$ in this expansion is
\begin{align*}
[\tau_{2r,r}^{\lambda}](\chi) &=\left(\frac{\lambda}{\pi}\right)^r\left. \frac{\partial^{r}}{\partial \eta^r}\frac{1}{\sqrt{1-\eta}}\right|_{(0,0)}\\
&={\left(\frac{\lambda}{\pi}\right)^r} \binom{2r}{r}r!4^{-r}.
\end{align*}

By Proposition \ref{prop:primitives}, we have \[[\tau_{2r,r}^{\lambda}](\sigma_{k,r}^\lambda)=\frac{\omega_{2n-k}}{(k-2r)!}[\tau_{2r,r}^{\lambda}](\pi_{k,r}^\lambda) = \frac{\omega_{2n-k}}{(k-2r)!}\delta_{k,2r},\] whence
\[
\begin{aligned}
[\tau_{2r,r}^\lambda]\left(\sum_j a_j \psi^{(j)}\right) &= a_r[\tau_{2r,r}^\lambda](\sigma_{2r,r}^\lambda) \\
&= a_r\omega_{2n-2r}.
\end{aligned}
\]
Hence
\[
a_r = {\left(\frac{\lambda}{\pi}\right)^r}\binom{2r}{r}\dfrac{r!}{4^r \omega_{2n-2r}}
\]
and the result follows.
\end{proof}

\subsection{Image of $\partial_{\lambda,\C}$ and $\partial_{\lambda,\R}$}
Next we describe the image of the operators $\partial_{\lambda,\C}$ and $\partial_{\lambda,\R}$, and we compute the preimage of any element belonging to them.

\begin{prop}
    Given any $\varphi = \sum_{k,r} a_{k,r}\sigma_{k,r}^\lambda \in\vlc$, we have $ \varphi\in \nm{im}\partial_{\lambda,\C}$ if and only if
    \begin{equation}\label{eq:imagen}
    \sum_{l = r}^{n-2r} a_{2l,r} \binom{n-2r}{l-r}\lambda^{n-l-r} = 0,\qquad\mbox{for }
   \quad 0 \leq 2r \leq n.
    \end{equation}
\end{prop}

\begin{proof}Note that $\varphi=\sum_{r}\varphi_r$ with $\varphi_r=\sum_{k} a_{k,r}\sigma_{k,r}^\lambda$ is the decomposition of $\varphi$ corresponding to $\vlc=\bigoplus_{r=0}^{\left\lfloor n/2\right\rfloor} \Inv{n}{r}$.
    By \eqref{eq:conmutan} and Proposition \ref{prop:ecuacionkerYlambda} we have $ \varphi\in \nm{im}\partial_{\lambda,\C}$ if and only if for every $r$
\begin{align}
   0= Z_{2n-4r,\lambda}(\varphi_r)&= \sum_{k=2r}^{2n-4r} a_{k,r} \binom{2n-4r}{k-2r}Z_{2n-4r,\lambda}(x^{k-2r}y^{2n-k-2r}) \\
&= \sum_{l=r}^{n-2r}a_{2l,r}\binom{2n-4r}{2l-2r}\binom{n-2r}{l-r}\lambda^{n-l-r}(2l-2r)!(2n-2l-2r)!\\
&=(2n-4r)!\sum_{l = r}^{n-2r} a_{2l,r} \binom{n-2r}{l-r}\lambda^{n-l-r}
\end{align}
where we used \eqref{eq:Z_monomio}.
\end{proof}

\begin{prop}
    Given  $\varphi = \sum_{k,r} a_{k,r}\sigma_{k,r}^\lambda \in\vlc$ satisfying \eqref{eq:imagen} we have
    \[
    \partial_{\lambda,\C}^{-1}(\{\varphi\})=\sum_{k,r} a_{k,r} J_{\lambda,\C}(P_{2n-4r,k-2r+1})+\nm{span}\{\psi_{2r}\: 0\leq 2r \leq n\}
    \]
    where $P_{m,l}$ is given by \eqref{def:Pk}.
\end{prop}
\begin{proof}
    This follows at once from Proposition \ref{prop:formulapreimagenYlambda} after decomposing $\varphi=\sum_r\varphi_r$ as in the previous proof.
\end{proof}

\begin{prop}
    The image of $\partial_{\lambda,\R}$ in $\vlr$ is the hyperplane $\mathcal H_\lambda^m$ generated by $\sigma_0^{\lambda},\dots,\sigma^{\lambda}_{m-1}$. Moreover
\begin{align}\label{eq:derivada_phik}
    \partial_{\lambda,\R}\phi^{k} =\dfrac{k! \omega_k}{\pi^{k}\omega_{m-k}}\sigma_{k-1}^\lambda, \qquad 1 \leq k \leq m,
\end{align}
    where $\phi^k=\sum_{j \geq 0}\left(\frac{\lambda}{4}\right)^j \tau_{k+2j}^\lambda$. In particular \[
    \partial_{\lambda,\R}^{-1}(\{\sigma^{\lambda}_{k-1}\})=\dfrac{\pi^{k}\omega_{m-k}}{k! \omega_k}\phi^k+\C\cdot\chi.
    \]      
\end{prop}
Recall from \cite[eq. (118)]{libroazul} that $\phi^k= \int_{G_{\lambda,\C}}\chi(\cdot \cap g \RR{m-k})dg$ where $dg$ is a properly normalized Haar measure on  $G_{\lambda,\C}$, and $\RR{m-k}$ is an $(m-k)$-dimensional totally geodesic submanifold in $\RR{m}$.
\begin{proof}
    By  \eqref{eq:sigmatau} and \eqref{eq:sigmatauvol}
\begin{align*}
    \phi^k = \sum_{j=0}^{\left\lfloor\frac{m-k-1}2\right\rfloor}\left(\dfrac{\lambda}{4\pi^2}\right)^j\dfrac{(k+2j)!\omega_{k+2j}}{\pi^{k}(m-k-2j)\omega_{m-k-2j}}\sigma_{k+2j}^\lambda+\left(\left(\frac\lambda4\right)^{\frac{m-k}2}\frac{m!\omega_m}{\pi^m}\sigma_m^\lambda\right), 
\end{align*} 
where the term between brackets appears only if $m-k$ is even. Using Proposition \ref{prop:der_sigma}, this yields \eqref{eq:derivada_phik}.  The rest of the statement follows.
\end{proof}

\begin{ob}
Equation \eqref{eq:derivada_phik} also follows from Theorem 4 in \cite{TAMSGil}.
\end{ob}

\end{document}